%% file: CDNP10.tex
\documentclass{amsart}
\usepackage{amsmath}
\usepackage{graphicx} 
\usepackage{amsfonts} 
\usepackage{amssymb}
\usepackage[font=small,format=hang,labelfont={sf,bf}]{caption}
\usepackage{color}
\usepackage{mathrsfs}
\usepackage{epsfig}
\usepackage{subfig}
\usepackage{url}
\usepackage{varioref}
\theoremstyle{plain}
\newtheorem{theorem}{Theorem}[section]

\newtheorem{lemma}[theorem]{Lemma}
\newtheorem{proposition}[theorem]{Proposition}
\newtheorem{corollary}[theorem]{Corollary}
\newtheorem{remark}[theorem]{Remark}

\theoremstyle{definition}
\theoremstyle{remark}
\numberwithin{equation}{section}

\newcommand{\ep}{\varepsilon}
\newcommand{\e}{\varepsilon}

\newcommand{\ffi}{\varphi}

\newcommand{\R}{\mathbb{R}}
\newcommand{\N}{\mathbb{N}}
\newcommand{\Z}{\mathbb{Z}}

\newcommand{\E}{\mathcal{E}}
\newcommand{\F}{\mathcal{F}}

\newcommand{\ud}{\;\mathrm{d}}

\newcommand{\weakly}{\rightharpoonup}           
\newcommand{\weakstar}{\stackrel{*}{\weakly}}   

\newcommand{\loc}{\mathrm{loc}}

\newcommand{\A}{\mathcal{A}}

\newcommand{\Am}{\mathcal{A}_{m_1,m_2}}

\newcommand{\Eab}{\E_K^{{c_{11}},{c_{22}}}}

\newcommand{\hs}{\mathcal{H}}
\newcommand{\I}{\mathcal{I}}
\newcommand{\J}{\mathcal{J}}
\newcommand{\dist}{\text{dist}}

\newcommand{\res}{\mathop{\hbox{\vrule height 7pt width .5pt depth 0pt
\vrule height .5pt width 6pt depth 0pt}}\nolimits}
\usepackage{latexsym}
\newcommand{\newatop}{\genfrac{}{}{0pt}{1}} 

\title[Two phases with cross and self attractive forces] {Ground states of a two phase model with cross and self attractive interactions}

\author[M. Cicalese]
{M. Cicalese}
\address[Marco Cicalese]{Zentrum Mathematik - M7, Technische Universit\"at  M\"unchen, Boltzmannstrasse 3, 85748 Garching, Germany}
\email[M. Cicalese]{cicalese@ma.tum.de}
\author[L. De Luca]
{L. De Luca}
\address[Lucia De Luca]{Zentrum Mathematik - M7, Technische Universit\"at  M\"unchen, Boltzmannstrasse 3, 85748 Garching, Germany}
\email[L. De Luca]{deluca@ma.tum.de}
\author[M. Novaga]
{M. Novaga}
\address[Matteo Novaga]{Dipartimento di Matematica, Universit\`a di Pisa, Largo Bruno Pontecorvo 5, 56127 Pisa, Italy}
\email[M. Novaga]{novaga@dm.unipi.it}
\author[M. Ponsiglione]
{M. Ponsiglione}
\address[Marcello Ponsiglione]{Dipartimento di Matematica ``Guido Castelnuovo'', Sapienza Universit\`a di Roma, P.le Aldo Moro 5, I-00185 Roma, Italy}
\email[M. Ponsiglione]{ponsigli@mat.uniroma1.it}

\begin{document}

\begin{abstract}
We consider a variational model for two interacting species (or phases), subject to cross and self attractive forces. We show existence and several qualitative properties of minimizers. Depending on the strengths of the forces, different behaviors are possible: phase mixing or phase separation with nested or disjoint phases. 
In the case of Coulomb interaction forces, we characterize the ground state configurations.
\end{abstract}

\maketitle
\tableofcontents

\section*{Introduction}

Models of two or more interacting species find applications in several fields of science, such as physics, chemistry and biology. To cite a few examples one may think about the formation of bacterial colonies in biology \cite{LB-JCR}, the self-assemble of nano-particles in physical chemistry \cite{LLPPTM}, the problem of two species group consensus \cite{EMV} as well as that of pedestrian dynamics \cite{CPT}. The basic feature of all these models is the presence of competing forces aiming to drive two phases towards different shapes. 

An interesting example of this phenomenon has been recently reported in \cite{LLPPTM}. There, it has been observed that, during the assembly process of two nano-scaled polyprotic macroions in a dilute aqueous solution, the system may be driven towards phase segregation as opposite to phase mixtures via a complex self-recognition mechanism involving multiple scales optimization. 

Far from thinking to propose realistic models for these complex mechanisms, we aim at reproducing such limit behaviors while keeping the number of parameters as small as possible. We propose and study a toy model for two interacting phases subject to self and cross attractive forces depending only on the distance between particles. Such a model may be introduced as follows. Two phases, represented by two subsets of $\R^N$, say $E_{1}$ and $E_{2}$, with masses $m_{1}$ and $m_2$ respectively, interact both with themselves and with the other phase trying to minimize an energy of the form
\begin{equation}\label{intro-energy}
\mathcal F(E_{1},E_2) = \sum_{i,j=1}^{2}J_{K_{ij}}(E_i,E_j). 
\end{equation}
Here
\begin{equation}\label{J}
J_{K_{ij}}(E_{i},E_{j}):=\int_{\R^{N}}\int_{\R^{N}} \chi_{E_i}(x)\,\chi_{E_j}(y)\,K_{ij}(x-y)\ud x\ud y
\end{equation} 
is a nonlocal interaction energy with interaction potential 
$K_{ij}:\R^{N}\to\R$.
Energy functionals of this type have been considered by many authors in the context of  nonlinear aggregation-diffusion problems, modeling biological swarming and crowd congestion (see \cite{WV,CCELM,CDFLS, DF, MKB, MRSV} and the references therein). 

In the present paper we initiate the analysis of the ground states of the energy functional $\F$ assuming that for $i,j\in\{1,2\}$ the interaction forces,   still having  different intensities,  obey the same nonlocal law. More precisely, we consider $K\in L^{1}_{\loc}(\R^N;\R)$ a non-increasing radially symmetric interaction potential and restrict our analysis to those $K_{ij}=c_{ij}K$. Moreover, we assume that the interactions are attractive, i.e., $c_{i,j}\le 0$. Without this assumption, different phenomena may appear, related to loss of mass at infinity. As a consequence, 
the minimization problem is in general ill-posed,  and requires specific cares. One possibility would consist in  adding some confinement conditions. In \cite{BKR},  the authors propose a different kind of problem: they focus on the case $c_{11}=c_{22}=1$, $c_{12}+c_{21}=-2$, fix $E_{1}$ and study the minimization of \eqref{intro-energy} as a function of $E_{2}$. They prove that such a problem admits a solution if and only if $m_{2}\leq m_{1}$. Similar threshold phenomena appear in energetic models for di-block copolymers, where a confining perimeter term and  a repulsive force compete \cite{BC, DNRV, FFMMM, Ju, KM1,KM2,LO} as well as in attractive/repulsive Lennard-Jones-type models (see e.g., \cite{CCP, CCH, CFT, KHP,KSUB,   SST} and the references therein). 

Let us go back to the case of attractive interactions $c_{ij}\leq 0$ considered in this paper. We will see that, also in this case, the minimization problem above is actually ill-posed.  Indeed, in Proposition \ref{propmass} and Theorem \ref{00}, we will show that if $|{c_{11}}|,|{c_{22}}|$ are small enough, any minimizing sequence wants to mix the two phases. We are then led to consider a relaxed version of the problem above where the notion of phase is weakened to allow local mixing.  Now the phases are described in terms of their densities $f_{1}, f_{2}\in L^{1}(\R^N;[0,1])$, so that $\int_{\R^{N}}f_{i}(x)\ud x=m_{i}$ and the functional becomes
\begin{equation}\label{intro-energy-e}
\E_{K}(f_{1},f_{2})=c_{11}\,J_K(f_1,f_1)+c_{22}\,J_{K}(f_2,f_2)+(c_{12}+c_{21})\, J_{K}(f_1,f_2),
\end{equation}
where $J_{K}(f_i,f_{j})$ has the same form of \eqref{J} with $K$ and  $f_{i}$ in place of $K_{ij}$  and $\chi_{E_{i}}$, respectively. 

For all  masses $m_i>0$ and all $c_{ij}\leq 0$, we prove existence of the minimizers of $\E_K$  under the constraint $f_{1}+f_{2}\leq 1$ (Theorem \ref{esist}). Such a constraint is inherited by the original problem, naturally arising from the relaxation procedure, but has also a clear physical meaning. Indeed, if we interpret the densities $f_{i}$ as proportional to the number of particles per unit volume on a certain mesoscopic ball of a lattice gas model, the condition reflects the fact that two particles are not allowed to occupy the same elementary cell. Note that for a slightly different problem in the one dimensional case, a similar existence result has appeared in \cite{KSX}.

In the case $c_{12} = c_{21}=0$, our problem reduces to two independent one-phase problems given by
$$
\min_{\newatop{f_i\in L^{1}(\R^N;[0,1])}{\int_{\R^N}f_i(x)\ud x=m_i}} c_{ii}\,J_{K}(f_i,f_i)\quad\text{ for }i=1,2.
$$
If $c_{ii}<0$, it is well known that the minimizer above is (the characteristic funciton of) a ball having mass equal to $m_i$ (see \cite{Riesz, LL} or Lemma \ref{riesz}).
Therefore,  we focus on the case $c_{12}+c_{21}<0$. Clearly, by the scaling and symmetry properties of the energy, it is not restrictive to assume $c_{12}=c_{21}=-1$. With this interaction term in the energy the geometry of the phases becomes a more delicate issue and it drastically depends on the strength of the interaction constants $c_{11}$ and $c_{22}$. On one hand, if the cross interaction forces prevail, {\it phase mixing} occurs, that is, a new phase appears which is a combination of the two pure phases. On the other hand, if one of the two self interaction forces is sufficiently strong, {\it phase segregation} occurs, with the presence of two pure phases which can be nested or adjacent, depending on the strength of the other force. 
The latter behavior is in a certain sense reminiscent of clusters of two phases in an infinite ambient phase, minimizing an inhomogeneous perimeter functional with surface tension depending on the two touching phases \cite{AB}. In this case the mixing of phases is impossible but, depending on the strength of the surfaces tensions, minimizers may exhibit disjoint or nested phases \cite{MN}. 

Our analysis focuses also on qualitative properties of  solutions. In some cases, we have determined the explicit geometry of the phases of the minimizers. Such an analysis is almost complete for the Coulomb interaction kernel.    


\begin{figure}[h]
\centering
{\def\svgwidth{360pt}
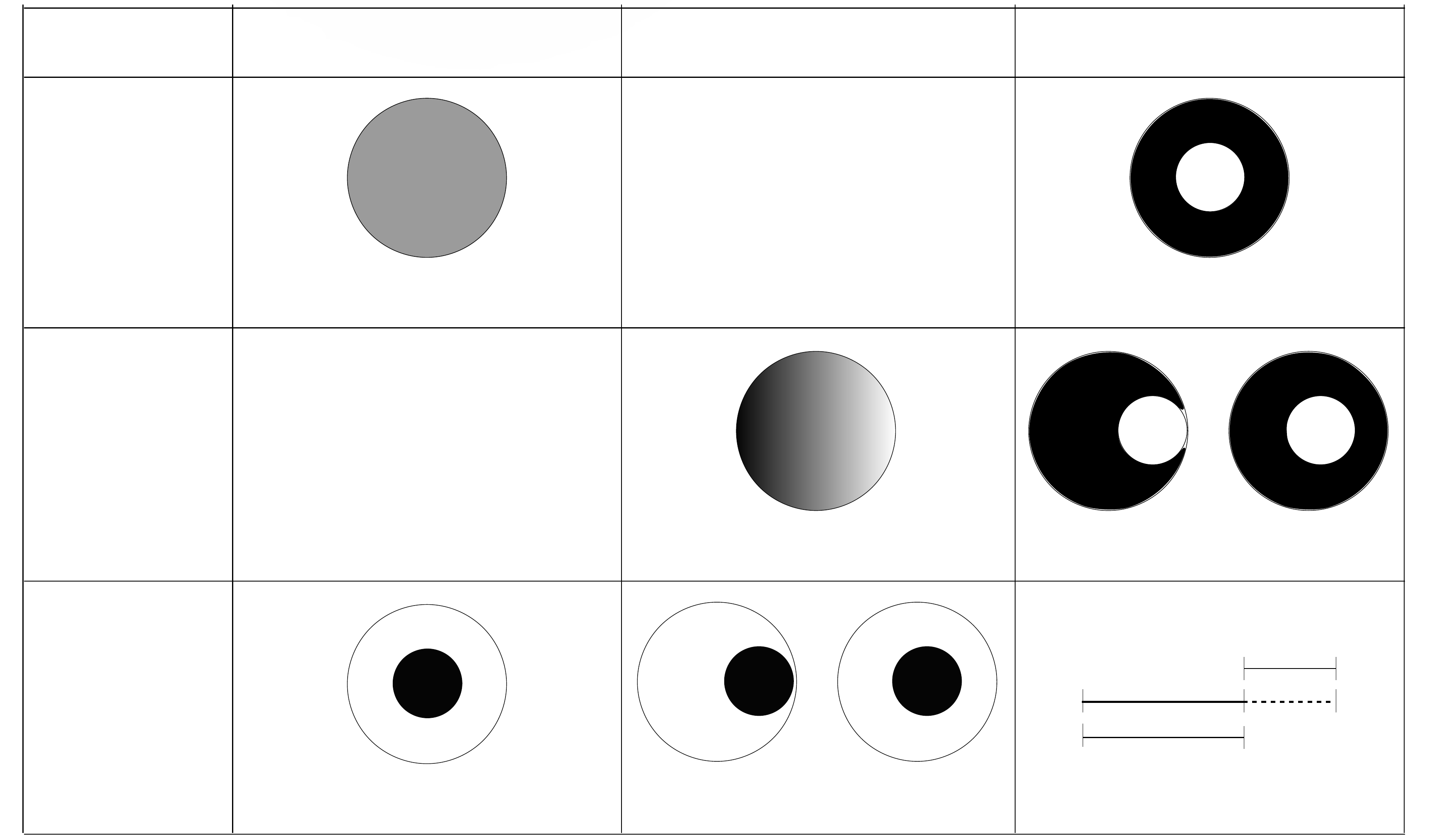}
\caption{
The phase $f_1$ is the black one, whereas the phase $f_2$ is white. The grey region represents the mixing of the two  phases. The gradational shaded ball in the central box represents the extremely degenerate character of minimizers for $c_{11}=c_{22}=-1$.
}
\label{fig1}
\end{figure}

We first describe the case of general kernels (see Figure \ref{fig1}).
First, consider the case $c_{11}+c_{22}> -2$, that we will call  the {\it weakly attractive case}. In this case,   the shape of minimizers is not explicit, except for $(c_{11}+1)\,m_1=(c_{22}+1)\,m_2$ and $K$ positive definite. If this occurs, the unique minimizer is given by $(f_1,f_2)=(\frac {m_1}{m_1+m_2}\chi_B,\frac {m_2}{m_1+m_2}\chi_B)$, where $B$ is a ball with $|B|=m_1+m_2$ (Proposition \ref{propmass}). 

The {\it strongly attractive case} $c_{11}+c_{22}\le -2$ (Theorem \ref{proptutto}) needs to be classified into the four subcases listed below. 
If $c_{11}=c_{22}=-1$, the problem is extremely degenerate, i.e., the minimizers are given by all the pairs $(f_1,f_2)$, with $f_1+f_2=\chi_{B}$.
If $c_{11}=-1$ and $c_{22}<-1$, then the minimizers of the problem are the pairs $(f_1,f_2)$, where ${f_1}+{f_{2}}=\chi_{B}$  and $f_2$  is (the characteristic function of) a ball  contained in $B$ (not necessarily concentric). If  $c_{22}<-1<c_{11}$, then the minimizer is unique and it is given by a ball and a concentric annulus around it.
Finally, for $c_{11},c_{22}<-1$, the minimizer is fully characterized only in the one dimensional case and it is given by the two tangent balls (namely segments).

As for the Coulomb interactions (see Figure \ref{fig2}), we have fully characterized the minimizers also in the weakly attractive case.
\begin{figure}[here]
\centering
{\def\svgwidth{360pt}
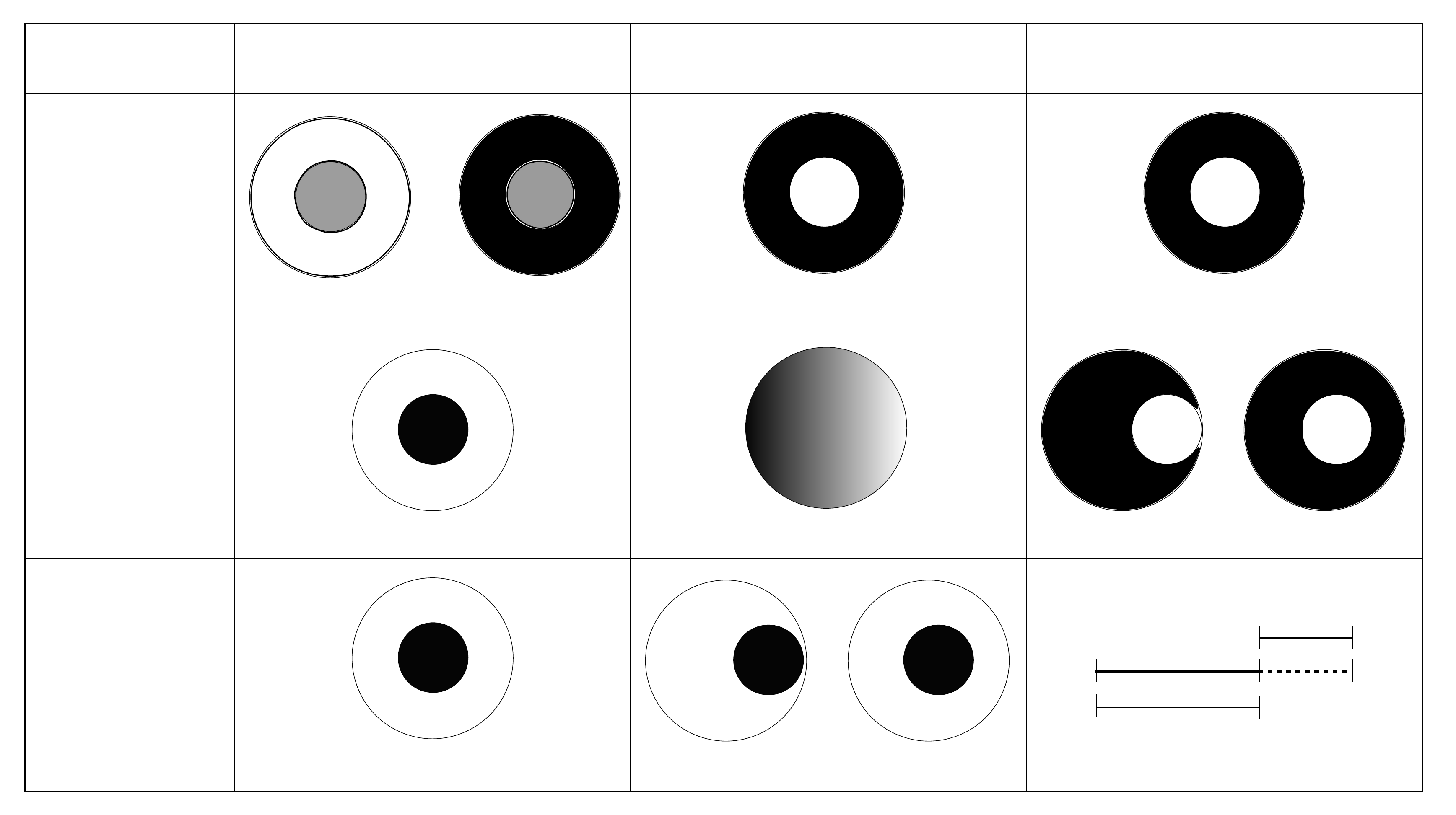}
\caption{
The phase $f_1$ is the black one,  the phase $f_2$ is white. The grey region represents the mixing of the two  phases. 
}
\label{fig2}
\end{figure}

We have proven (Theorem \ref{00}, Corollaries \ref{quadr1d} and \ref{minimoinquadrato}) that if $-1<c_{11}$, $c_{22}<0$ the minimizer is given by an interior ball in which $f_1$ and $f_2$ mix each other with specific volume fractions, according with their self attraction coefficients, and a concentric annulus where only the remaining homogeneous phase is present. If $c_{22}\le -1<c_{11}$, then the minimizer is unique and it is given by a ball and a concentric annulus around it.
In this respect,  for $c_{22}\le -1<c_{11}$ the solution is the same in the weakly and in the strongly attractive cases. 
  
Clearly, in the strongly attractive case the analysis done for general kernels applies in particular to the case of Coulomb interactions.
The shape of minimizers for $c_{11},c_{22}<-1$ is still open, but we can exclude the presence of a mixing phase (Proposition \ref{coroll1.5}).
The determination of the shape of the two phases in this case seems to be a challenging problem, that could be explored through numerical methods. 
Switching the roles of $c_{11}$, $c_{22}$, $f_1$, $f_2$  in the discussion above, 
the  description of minimizers extend  to all the other cases not explicitly mentioned.

We remark that the analysis for the Coulomb interaction kernel is much richer, since we can exploit  methods and tools of potential theory such as maximum principles. 
The characterization of minimizers in the weakly attractive case reduces to the case $c_{11} = c_{22} = 0$, considered in  Theorem \ref{00}. Even if   the two phases interact only through a cross attractive force,  this case turns out to be  non trivial.  The strategy to tackle this problem is based on a rearrangement argument that resembles the Talenti inequality. This is the content of  Lemma \ref{rearr0}, which establishes that, given a charge configuration $f$ which generates a potential $V$, one can rearrange the masses  on every superlevel of $V$, so that the new potential turns out to be greater than the radially symmetric rearrangement $V^*$ of $V$.  


%
%
%

The plan of the paper is the following. In Section \ref{secvar}
we introduce the nonlocal model and we prove existence and 
compactness of minimizers. In Section \ref{secqual} we show some qualitative properties of minimizers  and we characterize them explicitly in some  strongly attractive cases. Eventually, in Section \ref{seccoul} we study in detail the case of Coulomb interactions. 

\section{The variational problem}\label{secvar}

In this section we state our variational problem,  proving existence and some qualitative properties of the minimizers. 

\subsection{Description of the model}

We first introduce a functional modeling the interaction between two non-self-repulsive and mutually attractive species.

Let $N\in\N$  and let $K:\R^N\to\R$ be a non-increasing radially symmetric interaction potential, with  $K\in L^{1}_{\loc}(\R^N)$. 
For any pair of measurable sets $(A,B)$ with finite measure, we set
\begin{equation}\label{defJ}
J_K(A,B):=\int_{A}\int_{B}K(x-y)\ud x\ud y
\end{equation}
and we notice that, by the assumptions on $K$, the functional $J_K$ is well defined and takes values in $\R\cup\{-\infty\}$.

Given ${c_{ij}}\le 0$ for $i,j=1,2$ and  $m_1, \, m_2 >0$, we are interested in finding the minimizers of the functional
\begin{equation}\label{effe}
\mathcal{F}_K(E_1,E_2):={c_{11}}\,J_K(E_1,E_1)+{c_{22}}\,J_K(E_2,E_2)+(c_{12}+c_{21})\,J_K(E_1,E_2)
\end{equation}
among all the pairs of measurable sets $(E_1,E_2)$ with $|E_1|=m_1$ and $|E_2|=m_2$. Here $E_1$ and $E_2$ represent two species with masses $m_1$ and $m_{2}$ respectively, $c_{11},c_{22}$  the autointeraction and $c_{12}+c_{21}$ the cross-interaction coefficients.

As mentioned in the Introduction, for $c_{12}+c_{21}=0$ the problem decouples into two independent minimization problems, one for each phase. These are of the form 
$$
\min\left\{ -J_{K}(E,E): |E|=m\right\}.
$$
By the Riesz inequality \cite{Riesz} (see Lemma \ref{riesz}), such a one-phase problem is well known to be solved by a ball \cite{LL}.
As a consequence we focus on the case $c_{12}+c_{21}<0$ and furthermore, without loss of generality,  we set $c_{12}+c_{21}=-2$.
From Proposition \ref{propmass} and Theorem \ref{00}, it will follow that
 if $|{c_{11}}|,|{c_{22}}|$ are small enough, the minimum problem above does not admit in general a minimizer. Roughly speaking, the reason is that, in some cases, any minimizing sequence wants to mix the two phases. As a result, we are led to consider a {\it relaxed problem}. More precisely, 
according with \eqref{defJ}, for any $f_1,\,f_2\in L^{1}(\R^N)$ we set
$$
J_K(f_1,f_2):=\int_{\R^N}\int_{\R^N}f_1(x)\,f_2(y)\,K(x-y)\ud x\ud y.
$$
Then, we consider the functional $\Eab: L^1(\R^N;\R^+) \times L^1(\R^N;\R^+)\to\R\cup\{+\infty\}$ defined by
\begin{equation}\label{energy}
\Eab(f_1,f_2) = {c_{11}}\,J_K(f_1,f_1) + {c_{22}}\,J_K(f_2,f_2) -2\,J_K(f_1,f_2).
\end{equation}
We introduce the class of admissible densities $\Am$ defined by
 \begin{multline}\label{admissible}
 \A_{m_1,m_2}:= \left\{(f_1,f_2) \in L^1(\R^N;\R^+) \times L^1(\R^N;\R^+): \right.\\  
\left. \int_{\R^N} f_i(x) \ud x = m_i \text{ for } i=1,2, \, 
f_1(x)+f_2(x)\le 1\,\text{ for a.e. } x\in\R^N \right\}.
 \end{multline}
It is easy to see that for any $(f_1,f_2)\in\Am$ 
$$
\Eab(f_1,f_2)=\inf\liminf_{n\to\infty} \mathcal{F}_K(E^n_1,E^n_2),
$$
where the infimum is taken among all sequences $\{E_i^n\}$ ($i=1,2$) with $|E_i^n|=m_i$ and such that $\chi_{E_i^n}$ converge tightly to $f_i$.
We also observe that, if the kernel $K$ is bounded  at infinity, then  the energy is continuous with respect to tight convergence: if $f_i^n \weakstar f_i$ and $\|f_i^n\|_1\to \|f_i\|_1$ for $i=1,2$, then  $\Eab(f_1^n,f_2^n) \to
 \Eab(f_1,f_2)$.
  
For $i=1,2$, set $V_i:=f_i\ast K$, so that we can write
\begin{eqnarray}\nonumber
\Eab(f_1,f_2)&=&{c_{11}}\int_{\R^N}f_1(x)V_1(x)\ud x+{c_{22}}\int_{\R^N}f_2(x)V_2(x)\ud x\\ \nonumber
&&-2\int_{\R^N}f_1(x)V_2(x)\ud x\\
\label{energyV}
&=&{c_{11}}\int_{\R^N}f_1(x)V_1(x)\ud x+{c_{22}}\int_{\R^N}f_2(x)V_2(x)\ud x\\ \nonumber
&&-2\int_{\R^N}f_2(x)V_1(x)\ud x.
\end{eqnarray}
We now recall the definitions of the main classes of kernels we will focus on.
We say that the kernel $K$ is {\it positive definite} if 
\begin{eqnarray}\label{posdefk}
&&J_{K}(\ffi,\ffi)\ge 0\,\,\,\forall\ffi\in L^{1}(\R^N)\,\text{ and }\\ \nonumber
&&\text{$J_K(\ffi,\ffi)=0$ if and only if $\ffi= 0$ a.e. in $\R^N$.}
\end{eqnarray}
We denote by $K_{C_N}$ the Coulomb kernel in $\R^N$, defined by
\begin{equation}\label{colker}
K_{C_N}(x):= \left\{\begin{array}{ll}
-\dfrac 12\, |x|&\text{for }N=1,\\
\\
-\dfrac{1}{2\pi}\log|x|&\text{for }N=2,\\
\\
\dfrac{1}{(N-2)\,\omega_N}\,\dfrac{1}{|x|^{N-2}}&\text{for }N\ge 3,
\end{array}\right.
\end{equation}
where $\omega_{N}$ is the $N$-dimensional measure of the unitary ball in $\R^{N}$.
By definition, $-\Delta K_{C_N}=\delta_0$ for any $N$ so that $-\Delta V_i(x)=f_i(x)$.
In the following Remark we list some properties of the Coulomb kernels that will be useful in the following.

\begin{remark}\label{posdef}
\rm{By \cite[Theorem 1.15]{La} $K_{C_N}$ is positive definite for $N\ge 3$ but not for $N=1,2$. Nevertheless, by \cite[Theorem 1.16]{La}, for any $\ffi\in L^{1}(\R^2)$ with compact support and $\int_{\R^2}\ffi(x)\ud x=0$, we have 
$$
J_{K_{C_2}}(\ffi,\ffi)\ge 0,
$$
where equality holds true if and only if $\ffi(x)=0$ for a.e. $x\in\R^{2}$. Moreover, it is easy to see that the same result holds true also for $K_{C_1}$.}
\end{remark}


\subsection{First and second variations}\label{subfv}

For any given $(f_1,f_2)\in \Am$ set
\begin{equation}\label{defac}
G_i:= \{x\in\R^N:\,0<f_i(x)<1\}, \quad F_i:=   \{x\in\R^N:\,f_i(x)=1\}, \quad i=1,2.
\end{equation}
Moreover, we set
\begin{equation}\label{defs}
S:= \{x\in\R^N:\, f_1(x)+f_2(x)=1\}. 
\end{equation}

\begin{lemma}[{\bf First variation}]\label{primolemma}
Let $(f_1,f_2)$ be a minimizer of $\Eab$ in $\Am$. Let  $i, j\in\{1,2\}$ with $i\neq j$.
For any $\ffi_i,\psi\in L^{1}(\R^N;\R^+)$ with $\ffi_i=0$ a.e. in $\R^N\setminus (G_i\cup F_i)$, $\psi=0$ a.e. in $S$, and $\int_{\R^N}\ffi_i(x)\ud x=\int_{\R^N}\psi(x)\ud x$, we have
\begin{equation}\label{primaprimolemma}
\int_{\R^N}(\psi(x)-\ffi_i(x))({c_{ii}} V_i(x)-V_j(x))\ud x\ge 0.
\end{equation}

As a consequence, 
\begin{equation}\label{sommaconst}
{c_{ii}} V_i - V_j=\gamma_{i}\text{ a.e. in } G_i\setminus S.
\end{equation}
for some constant $\gamma_i\in\R$.


\end{lemma}

\begin{proof}
To simplify notation we prove the claim for $i=1$ and $j=2$. The proof of the other case can be obtained by switching $f_1$ with $f_2$ and ${c_{11}}$ with ${c_{22}}$. Without loss of generality, we assume $\ffi_1,\psi\in L^{\infty}(\R^N;\R^+)$.
For any $\ep>0$, we set 
$$
A^\ep:=\{x\in G_1 \cup F_1:\,\ep<f_1(x)\le 1\},\,\,\,B^\ep:=\{x\in \R^N:\,f_1(x)+f_2(x)< 1-\ep\}.
$$
It is easy to see that $A^\ep\nearrow (G_1\cup F_1)$, $B^\ep\nearrow (\R^N\setminus S)$.
Set 
$$
\ffi_{1}^{\ep}:=\frac{\int_{\R^N}\ffi_1(x)\ud x}{\int_{A^\ep}\ffi_1(x)\ud x}\cdot \ffi_1\res{A^\ep},\,\,\,\psi^{\ep}:=\frac{\int_{\R^N}\psi(x)\ud x}{\int_{B^\ep}\psi(x)\ud x}\cdot\psi\res{B^\ep};
$$
then $\|\ffi_1^{\ep}-\ffi_1\|_{L^1}\to 0$ and  $\|\psi^{\ep}-\psi\|_{L^1}\to 0$. 
For $t>0$ small enough,  $(f_1+t(\psi^{\ep}-\ffi_1^{\ep}),f_2)\in \Am$
and, since $(f_1,f_2)$ is a minimizer for $\Eab$, we have:
\begin{multline*}
0\le \lim_{t\to 0}\frac{\Eab(f_1+t(\psi^\ep-\ffi_1^\ep),f_{2})-\Eab(f_1,f_2)}{t}\\
=\int_{\R^N}2(\psi^\ep(x)-\ffi_1^\ep(x))\,({c_{11}} V_1(x)-V_2(x)) \ud x.
\end{multline*}
As $\ep\to 0$, we get the claim.

Finally, taking  $\ffi_1=\psi\equiv 0$ in $S\setminus G_1$ we are allowed to switch the roles of $\psi$ and $\ffi_1$ in \eqref{primaprimolemma}, obtaining \eqref{sommaconst}. 
\end{proof}

From now on, given any subset $E$ of $\R^N$, we will always assume that $E$ coincides with the set of the Lebesgue points of its characteristic function. In this way, $\partial E$ will be well defined and will always refer to this precise representative of $E$.

\begin{corollary}\label{coroll1}
Let $(f_1,f_2)$ be a minimizer of $\Eab$ in $\Am$. Then, for any $\ffi_1,\ffi_2\in L^{1}(\R^N;\R^+)$ with $\ffi_i=0$ a.e. in $\R^N\setminus (G_i\cup F_i)$ for $i=1,2$, and $\int_{\R^N}\ffi_1(x)\ud x=\int_{\R^N}\ffi_2(x)\ud x$, we have
\begin{equation}\label{scambio}
\int_{\R^N}(\ffi_2(x)-\ffi_1(x))(({c_{11}}+1) V_1(x)-({c_{22}}+1)V_2(x))\ud x\ge 0.
\end{equation}

In particular, for any $x_1\in \overline{G_1\cup F_1}$ and $x_2\in \overline{G_2\cup F_2}$, we have
\begin{equation}\label{ineq}
({c_{11}}+1)V_1(x_1)-({c_{22}}+1)V_2(x_1)\le({c_{11}}+1)V_1(x_2)-({c_{22}}+1)V_2(x_2).
\end{equation}

Moreover, 
\begin{equation}\label{diffconst}
({c_{11}}+ 1)V_1-({c_{22}} +1)V_2=\gamma
\qquad \text{ a.e.  in } G_1\cap G_2,
\end{equation}
for some constant $\gamma\in\R$.
\end{corollary}

\begin{proof}
Notice that if $\ffi_i=0$ a.e. in $\R^N\setminus G_i$ (for $i=1,2$), \eqref{scambio} is obtained
by summing \eqref{primaprimolemma} for $i=1, j=2$ with $\psi=\ffi_2$  and for $i=2, j=1$ with $\psi=\ffi_1$. 
To treat the general case, it is enough to consider the variation $(f_1+t(\ffi_2-\ffi_1),f_2-t(\ffi_2-\ffi_1))\in\Am$ for $t$ small enough. Then \eqref{scambio} and \eqref{ineq} are equivalent  to  the fact that the first variation of the energy is nonnegative.

Finally, taking $\ffi_1,\ffi_2\in L_{c}^{\infty}(\R^N;\R^+)$, with $\ffi_1=\ffi_2=0$ a.e. in $\R^N\setminus (G_1\cap G_2)$ and $\int_{\R^N}\ffi_1(x)\ud x=\int_{\R^N}\ffi_2(x)\ud x$, we have that \eqref{scambio} holds true also switching $\ffi_1$ 
with $\ffi_2$, whence we get \eqref{diffconst}.
\end{proof}

Arguing as in the proof of Corollary \ref{coroll1} (or using \eqref{primaprimolemma} and exploiting the continuity of $V_i$),
one can easily prove the following stationarity equations
for the boundaries of the two phases (see also \cite[Eqs. (1.2)--(1.4)]{RW}
for similar conditions in a related model for triblock copolymers).

\begin{corollary}\label{coroll2}
Let $(f_1,f_2)$ be a minimizer of $\Eab$ in $\Am$ and assume
that $f_i=\chi_{E_i}$ for some sets $E_i\subset\R^2$.
Then, the following equalities hold:
\begin{eqnarray}\label{eqqu}
{c_{11}} V_1 - V_2 &=& c_1 
\qquad \qquad \text{ on } \partial E_1\setminus \partial E_2
\\ \label{eqqa}
{c_{22}} V_2 - V_1 &=& c_2 
\qquad \qquad \text{ on } \partial E_2\setminus \partial E_1
\\ \label{eqqe}
({c_{11}}+ 1)V_1-({c_{22}} +1)V_2&=&c_1-c_2
\qquad \text{ on } \partial E_1\cap\partial E_1\,,
\end{eqnarray}
for some $c_1,\,c_2\in\R$.
\end{corollary}



\begin{lemma}[{\bf Second Variation}]\label{secondvar}
Let $(f_1,f_2)$ be a minimizer of $\Eab$ in $\Am$. Then for any $\ffi \in L^1(\R^N;\R)$ with $\ffi=0$ in $\R^N\setminus(G_1\cap G_2)$ 
and $\int_{\R^N}\ffi=0$, we have
\begin{equation}\label{equsecondvar}
({c_{11}}+{c_{22}}+2)\int_{\R^N}\int_{\R^N} 
K(x-y) \,  \ffi (x) \,  \ffi(y)   \ud x\ud y\ge 0.
\end{equation}
\end{lemma}

\begin{proof}
Without loss of generality assume that  $\ffi \in L^\infty (\R^N;\R)$.  Since $(f_1+t \ffi,f_2-t \ffi)\in\Am$ for $t$ small enough, the claim follows by the positiveness of the second variation at $(f_1,f_2)$, which is assumed to be a minimizer. The computations are left to the reader. 
\end{proof}
\subsection{Existence of minimizers}
Here we prove that for every ${c_{11}}, \, {c_{22}} \le 0$, the functional $\Eab$ defined in \eqref{energy}  admits a minimizer in $\Am$.



For any $m_1,m_2>0$, we set
$$
I^{{c_{11}},{c_{22}}}_{m_1,m_2}:=\inf _{(f_1,f_2)\in\mathcal{A}_{m_1,m_2}}\Eab(f_{1},f_{2})
$$
and we extend this definition to the case of possibly null masses, by setting
$$
I^{{c_{11}},{c_{22}}}_{m_1,m_2}:=\left\{\begin{array}{ll}
\min_{\newatop{f_i\in L^1(\R^N;[0,1])}{\int _{\R^N}f_i(x)\ud x=m_i}} {c_{ii}}\,J_K(f_i,f_i)&\text{ if }m_i>0\text{ and }m_j=0,\\
0&\text{ if }m_1=m_2=0.
\end{array}\right.
$$
%
%
%
%

The following two lemmas state  monotonicity and sub-additivity properties of the energy with respect to the masses $m_1$, $m_2$, for nonnegative kernels. Their proofs can be easily obtained exploiting the fact that the two phases attract each other: adding masses or moving back masses going to infinity decreases the energy. The details of the proofs are left to the reader.
\begin{lemma}\label{mono}
Assume that  $K(x)\ge 0$ for all $x\in\R^N$.  For any $m_1\ge\tilde m_1\ge0$ and $m_2\ge\tilde m_2\ge 0$  we have
$$
I^{{c_{11}},{c_{22}}}_{m_1,m_2}\le I^{{c_{11}},{c_{22}}}_{\tilde m_1,\tilde m_2}.
$$ 
Moreover, if $m_1,\,m_2>0$,  equality holds true if and only if $m_i=\tilde m_i$ for $i=1,2$.
\end{lemma}
\begin{lemma}\label{subadd}
Assume that  $K(x)\ge 0$ for all $x\in\R^N$. Let $\{m_1^l\},\,\{m_2^l\}$ be two nonnegative sequences such that  $0\le \tilde m_{i}:=\sum_{l\in\N}m_i^l<+\infty$ for $i=1,2$.
Then
\begin{equation}
\sum_{l\in\N}I^{{c_{11}},{c_{22}}}_{m_1^l,m_2^l}\ge I^{{c_{11}},{c_{22}}}_{\tilde m_1,\tilde m_2}.
\end{equation}
Moreover, if $\tilde m_1, \, \tilde m_2 >0$, then  equality holds true if and only if  $\tilde m_i^l\equiv 0$ for any $l\neq \bar l$, for some $\bar l\in\N$ and for $i=1,2$.
\end{lemma}

\begin{theorem}\label{esist}
Let ${c_{11}}, \, {c_{22}} \le 0$. Then,  the functional $\Eab$ defined in \eqref{energy} admits a minimizer in $\Am$. 
More precisely, let $\{(f_{1,n},f_{2,n})\}$ be a minimizing sequence. Then, there exists a sequence of  translations $\{\tau_n\}\subset\R^N$ such that 
 (up to a subsequence)
$f_{i,n}(\cdot - \tau_n) \to f_i$ tightly for some  $(f_1,f_2) \in \A_{m_1,m_2}$ which minimizes $\Eab$.
\end{theorem}

\begin{proof}
We distinguish between two cases. \\

{\it First case: $\lim_{|x|\to +\infty} K(x) = -\infty$}. 
For every $\e>0$ and for every pair of sets $A_{1,n}, \, A_{2,n}\subset \R^N$ such that 
$$
\int_{A_{i,n}} f_{i,n}(x) \ud x \ge \e,
$$ 
we have dist$(A_{1,n},  A_{2,n})\le  C$ for some $C$ independent of $n$; otherwise, we would clearly have $-J_K(f_{1,n},f_{2,n})\to + \infty$. 
As a consequence, by the triangular inequality we deduce that for every pair of sets $A_{i,n}, \, B_{i,n}\subset \R^N$ such that 
$$
\int_{A_{i,n}} f_{i,n}(x) \ud x \ge \e, \qquad \int_{B_{i,n}} f_{i,n}(x) \ud x \ge \e,
$$ 
we have dist$(A_{i,n},  B_{i,n})\le  C$  for some $C$ independent of $n$. As a result there exists $\{\tau_n\}\subset \R$ such that,  
up to a subsequence, $f_{i,n}(\cdot-\tau_n)$ tightly converge  to some $f_i$ in $L^1$.  
By the lower semicontinuity of $\Eab$ with respect to the tight convergence, we conclude that $(f_1,f_2)$ is a minimizer of $\Eab$ in $\Am$.\\

{\it Second case: $\lim_{|x|\to +\infty} K(x) +C=0$ for some $C\in\R$}.  For simplicity, we assume that  $C=0$, since additive constants in the kernel bring only an additive constant in the total energy. Set $Q_0:= [0,1]^N$, and for every $z \in \Z^N$, let $Q^{z}:=  z+ Q_0$ and $m^z_{i,n}:=\int_{Q^z} f_{i,n}(x) \ud x$.  
For any given $\e>0$, we set 
\begin{eqnarray*}
\mathcal I_{\e,n} := \{ z \in \Z^N: \, m^z_{i,n} \le \e, \, i=1,2\}, && \J_{\e,n} := \{ z \in \Z^N: \, \max_i \, m^z_{i,n} > \e\},\\
A_{\e,n}:= \bigcup_{z\in \I_{\e,n}} Q^z, &&     g^\e_{i,n}:=  f_{i,n}  \, \chi_{A_{\e,n}},\\
 E_{\e,n}:= \bigcup_{z\in \J_{\e,n}} Q^z, &&     f^\e_{i,n}:=  f_{i,n}  \, \chi_{E_{\e,n}}
\end{eqnarray*}
We first prove that
\begin{multline}\label{claim}
J_K(g^\e_{1,n}, f_{1,n}) + J_K(g^\e_{2,n},f_{2,n}) +  J_K(g^\e_{1,n},f_{2,n})
 + J_K(f_{1,n},g^\e_{2,n}) \le r(\e),
\end{multline}
where $r(\e) \to 0$ as $\e\to 0$.
We show only that $J_K(g^\e_{1,n},f_{2,n})<r(\ep)$ (the other cases being analogous). For every fixed $R \in\N$ we have
\begin{multline}\label{tante}
J_K(g^\e_{1,n},f_{2,n}) = \sum_{z \in \I_{\e,n}} \sum_{w \in \Z^N}  J_K(f_{1,n} \chi_{Q^z}, f_{2,n} \chi_{Q^w})\\
 = \sum_{z \in \I_{\e,n}, w \in \Z^N: |z-w|\le R}  J_K(f_{1,n} \chi_{Q^z}, f_{2,n} \chi_{Q^w})\\
+ \sum_{z \in \I_{\e,n}, w \in \Z^N: |z-w| > R}  J_K(f_{1,n} \chi_{Q^z}, f_{2,n} \chi_{Q^w}).
\end{multline}
Using that $K$ is integrable and by Riesz inequality (see Lemma \ref{riesz}), it is easy to see that
\begin{equation*}
J_K(f_{1,n} \chi_{Q^z}, f_{2,n} \chi_{Q^w}) \le h(m_{1,n}^z) \, m_{2,n}^w ,
\end{equation*}
where $h(t):=\int_{B^t} K(x) \, dx$ with $B^t$  the ball centered at the origin and with mass $t$, so that  $\lim_{t\to 0}  h(t) = 0$. We deduce that
the first addendum in \eqref{tante} tends to zero as $\e\to 0$ (for $R$ fixed). Moreover,   the second addendum  is bounded (uniformly with respect to $\e$) from above by a function $\omega(R)$,  such that $\omega(R)\to 0$ as  $R\to\infty$. This completes the proof of \eqref{claim}. 

By the mass constraints on $f_i$ we have that 
$\sharp \J_{\e,n} \le \frac{m_1+m_2}{\e}$. Therefore, up to a subsequence, we  can always write $\J_{\e,n} = \cup_{l=1}^{H_\e} J_{\e,n}^l$ for some ${H_\e}\le  \frac{m_1+m_2}{\e}$, where $J_{\e,n}^l$ are pairwise disjoint and satisfy:
\begin{itemize}
\item[(1)] for every $l$, $\mathrm{diam}(J_{\ep,n}^l)\le M_\e$ for some $M_\e\in \R$ independent of $n$;  
\item[(2)] for every $l_1\neq l_2$, $\dist(J_{\e,n}^{l_1}, J_{\e,n}^{l_2}) \to \infty$ as $n\to\infty$.
\end{itemize}
Notice that by \eqref{claim} we deduce that, for $\e$ small enough, $\J_{\e,n} \neq \emptyset$ and  $H_\e\ge1$ (otherwise $ I^{{c_{11}},{c_{22}}}_{m_1,m_2}$ would be zero). 
Set $ f^{\e,l}_{i,n}:=f_{i,n}^\ep\res\bigcup_{z\in J_{\ep,n}^l}Q^z$ for $i=1,2$ and for every $l=1, \ldots, H_\e$.  There exists a translation $\tau_{l,n}$ such that, up to a subsequence, 
$f^{\e,l}_{i,n}(\cdot - \tau_{l,n})$ converge tightly to some $f^{\e,l}_{i}$.
By \eqref{claim}, recalling that $\lim_{|x|\to +\infty} K(x) =0$   and  using the continuity of the energy with respect to the tight convergence,  we have
\begin{multline}\label{siusa}
\lim_n \Eab(f_{1,n},f_{2,n}) \ge \limsup_n \Eab(f^\e_{1,n},f^\e_{2,n}) - r(\e) \\
= \limsup_n \sum_{l=1}^{H_\ep} \Eab(f^{\e,l}_{1,n},f^{\e,l}_{2,n}) - r(\e) \geq\sum_{l=1}^{H_\ep} \Eab(f^{\e,l}_{1},f^{\e,l}_{2}) - r(\e). 
\end{multline}
Let now $\{\e_k\}$ be a decreasing sequence converging to zero as $k\to\infty$. 
We notice that $H_{\ep_k}$ is nondecreasing with respect to $k$ and then $ H_{\ep_k}\to H\in \N\cup{\infty}$. 
We can always choose the labels in such a way that $\{f_{i,n}^{\e_k,l} \}$, and so also their limits $f_{i}^{\ep_k,l}$, are monotone with respect to $k$. As a consequence, it is not restrictive to assume that the translation vectors $\tau_{l,n}$ are independent of $\ep$. By monotonicity,  $f_i^{\e_k,l}$ converge strongly in $L^1$ to some $f_{i}^{l}$ for any $1\le l \le H$ and  $i=1,\, 2$. 
By \eqref{siusa}  and the continuity of $\Eab$ with respect to the tight convergence, it follows that
\begin{equation}\label{quasiconcl}
I^{{c_{11}},{c_{22}}}_{m_1,m_2} =\lim_{n}\Eab(f_{1,n},f_{2,n})\ge \sum_{l=1}^{H} \Eab(f^{l}_{1},f^{l}_{2}).
\end{equation}

Let $m_i^l:=\int_{\R^N}f_{i}^l(x)\ud x$, then $\tilde m_i:=\sum_{l=1}^{H}m_{i}^l\le m_i$ for $i=1,2$.

By \eqref{quasiconcl} and lemmas \ref{subadd} and \ref{mono}, we get
$$
I^{{c_{11}},{c_{22}}}_{m_1,m_2} \ge \sum_{l=1}^{H} \Eab(f^{l}_{1},f^{l}_{2})
\ge \sum_{l=1}^{H} I^{{c_{11}},{c_{22}}}_{m^l_1,m^l_2} \ge I^{{c_{11}},{c_{22}}}_{\tilde m_1,\tilde m_2}\ge I^{{c_{11}},{c_{22}}}_{m_1,m_2};
$$
it follows that all the inequalities above are in fact equalities, $H=1$ and $\tilde m_i=m_i$, which concludes the proof. 
\end{proof}

\begin{remark}\label{morephase}
{\rm The problem considered in this paper could be generalized to the case of more than two phases, with mutual and self attractive interactions. We notice that, with minor changes, the existence of a solution for this generalized problem would follow along the  lines of  the proof of Theorem \ref{esist}.
}
\end{remark}

\begin{remark}\label{nonex}
\rm{Notice that in the case of ${c_{11}},{c_{22}}>0$ the functional $\Eab$ does not admit in general a minimizer in $\Am$. For instance, if ${c_{11}}>0$, then it is easy to see that, for $m_1$ large enough, any minimizing sequence $f_{1,n}$ for the first phase  tends to lose mass at infinity. As a consequence, $\Eab$ does not admit  a minimizer in $\Am$
  for $m_1$ large enough.
  
Moreover, assume that $K$ be a positive definite kernel
as in \eqref{posdefk}, and let ${c_{11}},{c_{22}}\ge 1$. Then, for any $(f_1,f_2)\in\Am$, we have
\begin{equation*}
\E_K^{c_{11},c_{22}}(f_1,f_2)=(c_{11}-1)\,J_{K}(f_1,f_1)+(c_{22}-1)\,\,J_{K}(f_2,f_2)+J_{K}(f_1-f_2,f_1-f_2)\ge 0.
\end{equation*}
It is easy to see that the infimum of $\E_K^{c_{11},c_{22}}$ is zero.
It follows that  $(f_1,f_2)$ is a minimizer of $\E_K^{c_{11},c_{22}}$ in $\Am$ if and only if  $m_1=m_2$, $c_{11}=c_{22}=1$
and  $f_i=f\in L^{1}(\R^N;[0,\frac 1 2])$ with $\int_{\R^N}f(x)\ud x=m_1=m_2$.

Finally,  if $m_1=m_2$ and $\max\{c_{11},c_{22}\}>1$, then still the energy does not admit a minimizer. Indeed any minimizer $(f_1,f_2)$ should satisfy  $f_1=f_2$ a.e. and  the  energy becomes
$$
(c_{11}+c_{22}-2)\,J_{K}(f,f),
$$
which does not admit a minimizer in the class of functions $f\in L^{1}(\R^N;[0,\frac 1 2])$ with $\int_{\R^N}f(x)\ud x=m_1=m_2$.
}
\end{remark}

\subsection{Compactness of minimizers}
Here we prove the compactness property of minimizers.
\begin{proposition}\label{compactmin}
Every minimizer $(f_1,f_2)$ of $\Eab$ in $\Am$ has compact support. 
\end{proposition} 

\begin{proof}
Assume by contradiction that $f_1$ has not compact support. Recalling the definition of $S$ in \eqref{defs}, we set $r:=(2\frac{m_1+m_2}{\omega_N})^{1/N}$ so that $|B_{r}\setminus S|>0$. For $R>0$ we now set  $\ffi^R_1:=f_1\chi_{(\R^N\setminus B_R)}$ and observe that for $R$ large enough we can find  $\psi^{R}\in L^{1}(\R^N;\R^+)$  such that $\psi^R\equiv 0$ in $S\cup(\R^N\setminus B_{r})$ and at the same time 
$\int_{\R^N}\ffi_1^R(x)\ud x=\int_{\R^N}\psi^R(x)\ud x>0$. 
Hence by \eqref{primaprimolemma} we have
$$
\int_{B_{r}}\psi^R(x)({c_{11}} V_1(x)-V_2(x))\ud x\ge \int_{\R^N\setminus B_R}\ffi_1^R(x)({c_{11}} V_1(x)-V_2(x))\ud x,
$$
or, equivalently,
\begin{multline*}
\int_{B_{r}}\psi^R(x)(|{c_{11}}| V_1(x)+V_2(x))\ud x\le \int_{\R^N\setminus B_R}\ffi_1^R(x)(|{c_{11}}| V_1(x)+V_2(x))\ud x.
\end{multline*}
Since $\int_{\R^N\setminus B_R}\ffi_1^R(x)\ud x=\int_{B_r}\psi^R(x)\ud x$, the previous inequality implies that
$$
\inf_{B_r}(|{c_{11}}| V_1+V_2)\le \sup_{\R^N\setminus B_R}(|{c_{11}}| V_1+V_2)
$$
which gives a contradiction for $R$ large enough. 
\end{proof}

\section{Qualitative properties of minimizers and some explicit solutions}\label{secqual}

In this Section we discuss some qualitative properties of the minimizers of $\Eab$, and for some specific choice of the coefficients ${c_{11}}$, ${c_{22}}$ we determine the explicit solution. 

\subsection{Some preliminary results}
First, we recall the classical Riesz inequality \cite{Riesz}. To this purpose, 
for any $m>0$ and $x_0\in\R^N$, we denote by $B^m(x_0)$  the ball centered in $x_0$ with $|B^m(x_0)|=m$ ($B^m$ if $x_0=0$). With a little abuse of notation, for any $x_0\in\R^N$ and  for any $f\in L^{1}(\R^N)$, we set $B^{f}(x_0):=B^{\|f\|_{L^1}}(x_0)$ ($B^{f}:=B^{\|f\|_{L^1}}$ if $x_0=0$).
Moreover, for every function $u\in L^1(\R^N;\R^+)$  we denote by $u^*$ the spherical symmetric nonincreasing rearrangement of $u$, satisfying 
\begin{equation}\label{simmetrizz}
\{u^{*} > t\} = B^{m_t} \text{ where } m_t:=|\{u>t\}| \qquad \text{ for all } t>0.
\end{equation}

\begin{lemma}[{\bf Riesz inequality}]\label{riesz}
Let $f,\,g\in L^{1}(\R^N;[0,1])$ with $\|f\|_ {L^1},\|g\|_{L^1}>0$. Then, 
\begin{multline*}
\int_{\R^N}\int_{\R^N}f(x)\,g(y)\,K(x-y)\ud x\ud y\le \int_{\R^N}\int_{\R^N}f^*(x)\,g^*(y)\,K(x-y)\ud x\ud y\\
\le\int_{\R^N}\int_{\R^N}\chi_{B^f}(x)\,\chi_{B^g}(y)\,K(x-y)\ud x\ud y ,
\end{multline*}
where the first inequality is in fact an equality if and only if $f(\cdot)=f^*(\cdot-x_0)$ and $g(\cdot)=g^*(\cdot-x_0)$ for some $x_0\in\R^N$, whereas the second inequality holds with the equality if and only if $f^*=\chi_{B^{f}}$ and $g^*=\chi_{B^{g}}$.
\end{lemma}

The following lemma states that, for ${c_{11}}=0$, there exists a minimizer $(\tilde f_1,f_2)$ such that $\tilde f_1+f_2=1$ on the support of $\tilde f_1$. Its proof can be easily obtained by \eqref{energyV}; the details are left to the reader.

\begin{lemma}[{\bf Superlevels}]\label{corsuper}
Let $(f_1,f_2)$ be a minimizer of $\E_K^{0,{c_{22}}}$ in $\Am$. Then, there exists a unique $t>0$ such that the pair $(\tilde f_1,f_2)$ is still a minimizer of $\E_K^{0,{c_{22}}}$ in $\Am$ if and only if $\tilde f_1\in L^1(\R^N;\R^+)$ satisfies (i), (ii) and (iii) below:
\begin{itemize}
\item[(i)]$\int_{\R^N} \tilde f_1(x) \ud x = m_1$;
\item[(ii)] $\tilde f_1(x)= 1-f_2(x)$ if  $V_2(x)> t$;
\item[(iii)] $\tilde f_1(x) = 0$ if   $V_2(x) < t$.
\end{itemize}
Moreover, if $|\{ V_2=t\}|=0$, then $\tilde f_1$ is uniquely determined, and clearly $f_1=\tilde f_1$.

A similar statement holds true for the case ${c_{22}}=0$.
\end{lemma}
We recall that the sets $G_i$ are defined in \eqref{defac}.
\begin{corollary}\label{concentra}
Let $(f_1,f_2)$ be a minimizer for $\mathcal{E}_K^{0,{c_{22}}}$  in $\Am$. Then,
for any measurable set 
$E_{1}\subset G_1\setminus G_2$  
with $|E_1|=\int_{G_1\setminus G_2} f_1(x)\ud x$,   
the function 
$$
\tilde f_1(x):=\left\{\begin{array}{ll}
\chi_{E_1} &\text{ if }x\in  G_1\setminus G_2\\
f_1(x)&\text{otherwise in } \R^N, 
\end{array}\right.
$$
satisfies
$$
\mathcal{E}_K^{0,{c_{22}}}(\tilde f_1, f_2) \le\mathcal{E}_K^{0,{c_{22}}}(f_1,f_2).
$$

A similar statement holds true in the case ${c_{22}}=0$.

\end{corollary}

\subsection{The strongly attractive case {${c_{11}} + {c_{22}} \le -2$}}

In the following theorem we characterize the minimizers for every ${c_{11}}, \, {c_{22}} $ such that ${c_{11}} +{c_{22}} \le -2$ and 
$\max \{{c_{11}}, {c_{22}}\} \ge -1$.
\begin{theorem}\label{proptutto}
Let ${c_{11}}+{c_{22}}\le -2$. The following statements hold true.
\begin{itemize}
\item[(i)] if ${c_{11}}={c_{22}}=-1$, then $(f_1,f_2)$ is a minimizer of $\Eab$ in $\Am$ if and only if $f_1+f_2=\chi_{B^{m_1+m_2}(x_0)}$, for some $x_0\in\R^N$;
\item[(ii)] if ${c_{11}}=-1$ and ${c_{22}}<-1$, then $(f_1,f_2)\in\Am$ is a minimizer  of $\Eab$ in $\Am$  if and only if $f_1+f_2=\chi_{B^{m_1+m_2}(x_0)}$ for some $x_0\in\R^N$, and $f_2=\chi_{B^{m_2}(y_0)}$ for some $y_0\in\R^N$  with $B^{m_2}(y_0)\subset B^{m_1+m_2}(x_0)$;
\item[(ii')] if ${c_{22}}=-1$ and ${c_{11}}<-1$, then $(f_1,f_2)\in\Am$ is a minimizer  of $\Eab$ in $\Am$  if and only if $f_1+f_2=\chi_{B^{m_1+m_2}(x_0)}$
for some $x_0\in\R^N$ and $f_1=\chi_{B^{m_1}(y_0)}$ for some $y_0\in\R^N$ with $B^{m_1}(y_0)\subset B^{m_1+m_2}(x_0)$;
\item[(iii)] if ${c_{22}}<-1$ and $-1<{c_{11}}\le 0$, then $(f_1,f_2)\in\Am$ is a minimizer  of $\Eab$ in $\Am$  if and only if $f_1+f_2=\chi_{B^{m_1+m_2}(x_0)}$ and $f_2=\chi_{B^{m_2}(x_0)}$ for some $x_0\in\R^N$;
\item[(iii')] if ${c_{11}}<-1$ and $-1<{c_{22}}\le 0$, then $(f_1,f_2)\in\Am$ is a minimizer  of $\Eab$ in $\Am$  if and only if $f_1+f_2=\chi_{B^{m_1+m_2}(x_0)}$ and $f_1=\chi_{B^{m_1}(x_0)}$ for some $x_0\in\R^N$.
\end{itemize}
\end{theorem}
\begin{figure}[h]
\centering
{\def\svgwidth
{300pt}
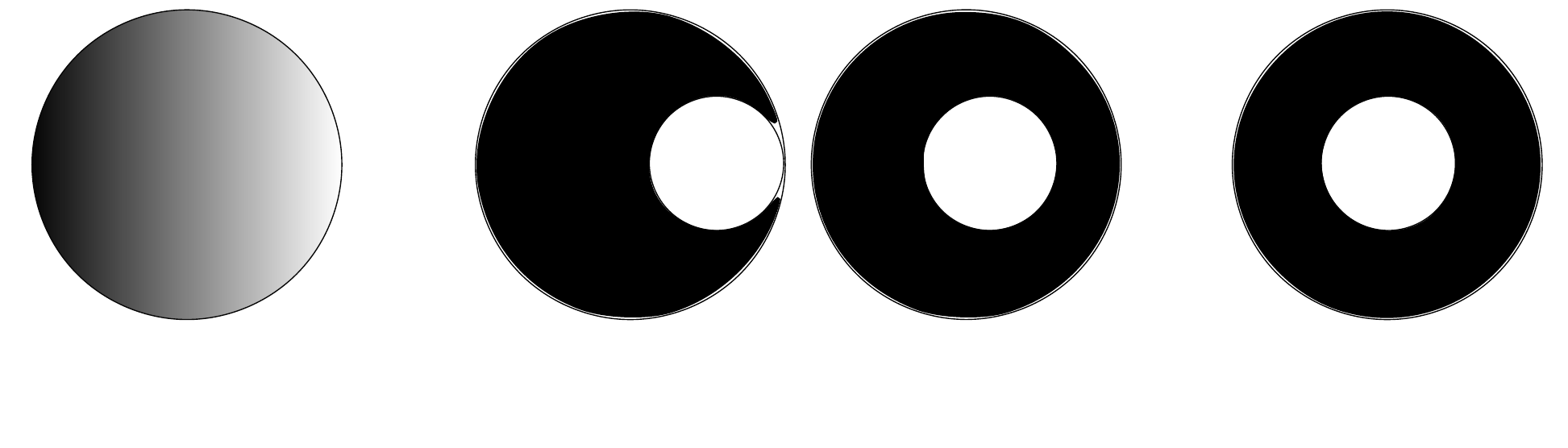}
\vskip -0.2cm
\caption{The phase $f_1$ is the black one, whereas the phase $f_2$ is white. The first cartoon represents the minimizers  in (i). In this case,  all the configurations $(f_1,f_2)$ such that $f_1+f_2=\chi_{B^{m_1+m_2}}$ are minimizers of the energy. The second and the third cartoons are two examples of minimizers in the case (ii). The last cartoon is the unique minimizer in the case (iii). Minimizers in cases (ii') and (iii') can be obtained by the balls above switching the balck parts with the white ones.
}
\end{figure}

\begin{proof}
We prove only (i), (ii), and (iii), being the proofs of (ii') and (iii') 
the same as to the ones of (ii) and (iii), respectively.

It is easy to see that
\begin{multline*}
\Eab(f_1,f_2)
={c_{11}}\,J_K(f_1+ f_2,f_1+f _2)-2({c_{11}}+1)\,J_K(f_2,f_1+f_2)\\
+({c_{11}}+{c_{22}}+2)\,J_K(f_2,f_2).
\end{multline*}
The claim (iii) follows immediately by applying Lemma \ref{riesz} to each of the three addenda above.
Moreover, 
\begin{align*}
&\mathcal{E}_K^{-1,-1}(f_1,f_2)=-J_K(f_1+f_2,f_1+f_2),\\
&\mathcal{E}_K^{-1,{c_{22}}}(f_1,f_2)=-J_K(f_1+f_2,f_1+f_2)+({c_{22}}+1)\,J_K(f_2,f_2)
\end{align*}
and hence (i) and (ii) easily follow by applying once again Lemma \ref{riesz}.
\end{proof}

The next proposition gives a characterization for $N=1$ of the minimizer of $\mathcal{E}_K^{{c_{11}},{c_{22}}}$ in the case ${c_{11}}, \, {c_{22}} < -1$ which is left open in Theorem \ref{proptutto}.

\begin{proposition}\label{oned}
Let $N=1$ and ${c_{11}},\,{c_{22}}<-1$. Then 
$$
(f_1,f_2)=(\chi_{[-m_1,0]},\chi_{[0,m_2]})\text{ and }(f_1,f_2)=(\chi_{[0,m_1]},\chi_{[-m_2,0]}),
$$
are (up to a translation) the unique minimizers of $\Eab$ in $\Am$.
\end{proposition}

\begin{proof}
It is easy to see that for any $(f_1,f_2)$
$$
\Eab(f_1,f_2)=\mathcal{E}_K^{{c_{11}},-1}(f_1,f_2)+({c_{22}}+1)J_K(f_2,f_2);
$$
since the second addendum is minimized when $f_2$ is the characteristic function of an interval, the claim follows by Theorem \ref{proptutto} (ii').
\end{proof}

\begin{remark}\label{cong}
{
\rm
In the general multi-dimensional case,  we do not know the explicit form of the minimizers if ${c_{11}},\, {c_{22}}< -1$.  One could guess that $f_i$ are characteristic functions as in the Coulomb case considered in Proposition \ref{coroll1.5} where, by means of first variation techniques, we can exclude that the solution is given by two tangent balls as well as by a ball and a concentric annulus around it. A natural issue to consider is then the asymptotic behaviour of minimizers  for ${c_{11}},\,{c_{22}}$ which tend to the boundary (and at infinity) of the region $\{{c_{11}}, \, {c_{22}} < -1\}$. In fact, there are many interesting limits that one could study:
\begin{itemize}
\item[(1)] ${c_{11}}< -1$, ${c_{22}}_n \nearrow - 1$ . Let $(f_1^{n},f_2^{n})\in\Am$ ba a minimizer of $\mathcal{E}^{{c_{11}},{c_{22}}_n}$ in $\Am$. Notice that the limit problem does not admit a unique solution. Nevertheless, we expect that, up to a unique translation, $f^n_1$ and $f^n_1+f^n_2$ converge strongly in $L^1$ to characteristic functions of two innerly tangent balls. Indeed, this is the minimizer for ${c_{11}}, {c_{22}} <1$,  among the family of pairs of nested balls.
\item[(2)] ${c_{11}},{c_{22}}< -1$, ${c_{11}}, {c_{22}} \nearrow - 1$.  In this case the limit problem is the most degenerate one for which it seems difficult to have a clear guess.
\item[(3)] ${c_{11}} < -1$, $ {c_{22}}_n \to - \infty$. In this case we expect that the second phase tends to a ball, while the first phase tends to the characteristic function of a set which is not a ball.
\item[(4)] ${c_{11}}, \, {c_{22}} \to -\infty$:  In this case we have that the two phases converge to two tangent balls. This is precisely the content of Proposition \ref{tangball} below. 
\end{itemize}
}
\end{remark}

\begin{proposition}\label{tangball}
Let $\{{c_{11}}_n\},\{{c_{22}}_n\}\subset\R$ be such that ${c_{11}}_n, \, {c_{22}}_n\to-\infty$.  For any $n\in\N$, let $(f_1^{n},f_2^{n})\in\Am$ be a minimizer of $\mathcal{E}_K^{{c_{11}}_n,{c_{22}}_n}$ in $\Am$. Then, up to a unique translation, $f^n_1,\,f^n_2$ converge strongly in $L^1$ to characteristic functions of two tangent balls, i.e., there exists a family of translations $\{\tau_n\}$ and a unitary vector $\nu\in\R^N$, such that 
$$
f_1^n(\cdot-\tau_n)\to \chi_{B^{m_1}},\quad f_2^n(\cdot-\tau_n)\to \chi_{B^{m_2}(r\,\nu)}, \quad\text{with $r:=(\textstyle {\frac{m_1}{\omega_N}})^{\frac 1 N}$.}
$$
\end{proposition}

\begin{proof}
First, notice that there exists a constant $C$ such that 
$$
-2 J_K(f_1^n,f_2^n) \ge C, \quad {c_{11}}_n\, J_K(f_1^n,f_1^n) \ge {c_{11}}_n\, I_{m_1,0}^{-1,0},  
\quad {c_{22}}_n\,J_K(f_2^n,f_2^n) \ge {c_{22}}_n \,I_{0,m_2}^{0,-1},
$$
so that
\begin{equation}\label{eq1}
I_{m_1,m_2}^{{c_{11}}_n,{c_{22}}_n} =\E_K^{{c_{11}}_n,{c_{22}}_n}(f_1^n,f_2^n) \ge {c_{11}}_n\,I_{m_1,0}^{-1,0} + {c_{22}}_n  \,I_{0,m_2}^{0,-1} + C.
\end{equation}
On the other hand, 
\begin{equation}\label{eq2}
I_{m_1,m_2}^{{c_{11}}_n,{c_{22}}_n} \le \E_K^{{c_{11}}_n,{c_{22}}_n}(\chi_{B^{m_1}},\chi_{B^{m_2}(r\,\nu)}) = 
{c_{11}}_n\,I_{m_1,0}^{-1,0} + {c_{22}}_n \,I_{0,m_2}^{0,-1} + C,
\end{equation}
which, togehter with \eqref{eq1}, yields 
$$
J_K(f_1^n,f_1^n) \to I_{m_1,0}^{-1,0}, \qquad J_K(f_2^n,f_2^n) \to I_{0,m_2}^{0,-1}.
$$ 
Therefore, by Theorem \ref{esist} applied to $I_{m_1,0}^{-1,0}$ and $I_{0,m_2}^{0,-1}$, there exist two sequences of translations $\{\tau_i^n\}$ (for $i=1,2$)  such that 
$$
f_1^n(\cdot - \tau_1^n) \to  \chi_{B^{m_1}}, \qquad f_2^n(\cdot - \tau_2^n) \to  \chi_{B^{m_2}}\qquad\text{strongly in }L^1.
$$
It remains to prove that $|\tau_1^n-\tau_2^n|\to r$ as $n\to\infty$. Set   
$$
\lambda_n :=\frac{|\tau_1^n-\tau_2^n|}{r}.
$$
Notice that $\liminf_{n\to\infty}\lambda_n\ge 1$ (otherwise, for $n$ large, $f_1^n$ and $f_2^n$ would be close in $L^1$ to characteristic functions of two intersecting balls, so that $(f_1^n,f_2^n)$ would not be admissible).
Up to a subsequence, we can assume that $\limsup_{n\to\infty}\lambda_n=\lim_{n\to\infty}\lambda_n=:\lambda$, with $\lambda\ge 1$. Then, set 
$$
\tilde f_1^n:=\chi_{B^{m_1}(\tau_1^n)},\quad \tilde f^n_2:=\chi_{B^{m_2}(\tau_2^n)};
$$ 
notice that $\|\tilde f_i^n- f_i^n\|_{L^1}\to 0$ as $n\to\infty$ for  $i=1,2$. Then, by the lower semicontinuity property of $J_K$ with respect to the strong $L^1$ convergence, we get 
$$
\liminf_n J_K(f_1^n,f_2^n)-J_K(\tilde f_1^n,\tilde f_2^n) \ge 0
$$
We conclude 
\begin{multline*}
I^{{c_{11}}_n,{c_{22}}_n}_{m_1,m_2}\ge \mathcal{E}_K^{{c_{11}}_n,{c_{22}}_n}(\tilde f_1^n,\tilde f_2^n)+\rho(n)\ge \E_K^{{c_{11}}_n,{c_{22}}_n}(\chi_{B^{m_1}},\chi_{B^{m_2}( r \nu)})+\rho(n)+\omega(\lambda_n),
\end{multline*}
where $\rho(n)\to 0$ as $n\to\infty$ and $\omega:[1,+\infty)\to [0,+\infty)$ is an increasing function vanishing at $1$.
By minimality it easily follows that $\lambda=1$ and hence the claim.
\end{proof}

\subsection{The weakly attractive case {${c_{11}} +  {c_{22}} > -2$}}
Here we will consider the  case ${c_{11}} +  {c_{22}} > -2$, and we will characterize the solution only for the purely weakly attractive case $0\ge {c_{11}}, \,  {c_{22}} > -1$  with
$({c_{11}}+1)m_1=({c_{22}}+1)m_2$. Moreover,  
 we will assume that $K$ is positive definite, according to definition \eqref{posdef}.
Notice that this implies that the functional 
$J_K(\ffi,\ffi)$
is strictly convex.

\begin{lemma}\label{fact}
Let $K$ be positive definite.
 For any $-1<c<1$ and for any $m>0$, the (unique up to a translation) minimizer of $ \mathcal{E}_K^{c,c}$ in $\mathcal{A}_{m,m}$ is given by the  pair $(f_1^0,f_2^0)=\big(\frac{1}{2}\chi_{B^{2m}}, \frac{1}{2}\chi_{B^{2m}}\big)$.

\end{lemma}

\begin{proof}
Let $(f_1,f_2)\in\mathcal{A}_{m,m}$.
We first notice that the convexity of the functional $J_K(f,f)$ immediately implies that 
\begin{equation}\label{primaequ}
\textstyle J_K(f_1,f_2)=2 J_K(\textstyle \frac{f_1+f_2}{2},\frac{f_1+f_2}{2})-
\frac{J_K(f_1,f_1)}{2}-\frac{J_K(f_2,f_2)}{2}\le J_K(\frac{f_1+f_2}{2},\frac{f_1+f_2}{2}).
\end{equation}
Moreover
$$
\textstyle \mathcal{E}_K^{c,c}(f_1,f_2)=c\, J_K(f_1+f_2,f_1+f_2)-2\,(1+c)\,J_K(f_1,f_2),
$$
which, together with \eqref{primaequ}, yields
\begin{multline*}
\textstyle
\mathcal{E}_K^{c,c}(f_1,f_2) \ge c\,J_K(\frac{f_1+f_2}{2}+\frac{f_1+f_2}{2},\frac{f_1+f_2}{2}+\frac{f_1+f_2}{2})  
\\      
\textstyle
-2(1+c)\,J_K(\frac{f_1+f_2}{2},\frac{f_1+f_2}{2})
= \mathcal{E}_K^{c,c}(\frac{f_1+f_2}{2},\frac{f_1+f_2}{2}),
\end{multline*}
where in the inequality we have also used  that $c+1>0$.
By  the strict convexity of $J_K(f,f)$, the inequality is strict whenever $f_1\neq f_2$. We deduce that   $f_1=f_2=\frac{f_1+f_2}{2}=:f$.
Since
$
\mathcal{E}_K^{c,c}(f,f)=2(c-1)\,J_K(f,f), 
$
 by Lemma \ref{riesz}, we conclude that $\mathcal{E}_K^{c,c}(f_1,f_2)$ attains its unique minimum when $f_1=f_2=\frac{1}{2}\chi_{B^{2m}}$.
\end{proof}

Let us introduce the coefficients $a_i$ (depending on ${c_{11}}$ and ${c_{22}}$) which represent the volume fractions of the two phases where they mix:
\begin{equation}\label{ai}
a_1:=:=\frac{{c_{22}}+1}{{c_{11}}+{c_{22}}+2}, \qquad a_2:= \frac{{c_{11}}+1}{{c_{11}}+{c_{22}}+2}. 
\end{equation}
Notice that, if $({c_{11}}+1)m_1=({c_{22}}+1)m_2$, then
$$
a_1= \frac{m_1}{m_1+m_2}, \qquad a_2=\frac{m_2}{m_1+m_2}.
$$
\begin{proposition}\label{propmass}
Let $-1<{c_{11}}, \, {c_{22}}\le 0$. If $({c_{11}}+1)m_1=({c_{22}}+1)m_2$, then the (unique up to a translation)
minimizer of $\Eab$ in $\Am$ is given by the pair
$$
(f_1,f_2)=(a_1 \chi_{B^{m_1+m_2}},a_2 \chi_{B^{m_1+m_2}}).
$$
\end{proposition}

\vskip -0.5cm
\begin{figure}[h]
\centering
{\def\svgwidth{90pt}
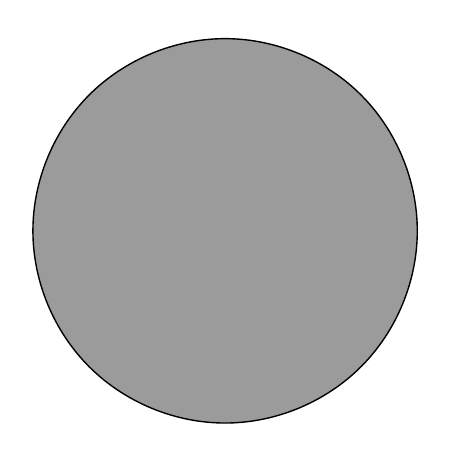}
\vskip -0.2cm
\caption{The phase $f_1$ is black and the phase $f_2$ is white.  Under the assumptions of Proposition \ref{propmass}, the minimzer is given by a ball where the two phases mix each other. The mixing is represented by the grey color.
}
\end{figure}

\begin{proof}
By Lemma \ref{fact} we get directly the claim in the case ${c_{11}}={c_{22}}$, since by assumption this implies $m_1=m_2$.

We now prove the result in the general case.
For any $(f_1,f_2)\in\Am$, we set

\begin{equation}\label{formulah}
h_1:=(\textstyle 1+\frac{{c_{11}}}{2})f_1-\frac{{c_{22}}}{2}f_2,\quad h_2:=-\frac{{c_{11}}}{2}f_1+(1+\frac{{c_{22}}}{2})f_2.
\end{equation}
It is easy to see that $h_1,h_2\ge 0$, $h_1+h_2=f_1+f_2\le 1$ and,
by assumption,
$$
\int_{\R^N}h_1(x)\ud x =\int_{\R^N}h_2(x)\ud x =\frac{m_1 + m_2}{2}=:m.
$$
By straightforward computations it follows that, setting $c:=\frac{{c_{11}}{c_{22}}}{2-{c_{11}}{c_{22}}}$,
\begin{equation*}
\Eab(f_1,f_2)
=\frac{2-{c_{11}}{c_{22}}}{2+{c_{11}}+{c_{22}}}\,\mathcal{E}_K^{c,c}(h_1,h_2).
\end{equation*}
Notice that, since $-1<{c_{11}},{c_{22}}<0$, we have that $0<c<1$ and $\frac{2-{c_{11}}{c_{22}}}{2+{c_{11}}+{c_{22}}}>0$; therefore, $(f_1,f_2)$ minimizes $\Eab$ (in $\Am$) if and only if $(h_1,h_2)$ minimizes $\mathcal{E}_K^{c,c}$ in $\mathcal{A}_{m,m}$.
By Lemma \ref{fact}, the unique minimizer of $\mathcal{E}_K^{c,c}$ in $\mathcal{A}_{m,m}$ is given by
$(h_1,h_2)=(\frac{1}{2}\chi_{B^{2m}},\frac{1}{2}\chi_{B^{2m}})$. Hence the claim for ${c_{11}}\neq{c_{22}}$ follows by \eqref{formulah}.
\end{proof}

\begin{remark}
{\rm
Proposition \ref{propmass} establishes that, for very special coefficients ${c_{11}}$ and ${c_{22}}$ depending on the masses $m_1$, $m_2$, the minimizer is given by a homogenous density that mixes the two phases with specific volume fractions. The proof is based on the convexity of $J_K$. One may wonder whether, under this assumption, the result still holds for generic ${c_{11}}$ and ${c_{22}}$. We will see that this is not the case even for the Coulomb kernel (see Corollary \ref{quadr1d} and Theorem \ref{00}). 
}
\end{remark}



\section{The Coulomb  kernel}\label{seccoul}
Through this section we will assume that $K=K_{C_N}$ is the Coulomb kernel defined in \eqref{colker}.
We will provide the explicit form of the solutions for all the choices of the (nonpositive) parameters ${c_{11}},\, {c_{22}}$, 
except when they are both strictly less than $-1$ in which case we will only be able to say that $f_i$ are characteristic functions of sets.

  
\subsection{Consequences of the first variation}
We specialize the results of section \ref{subfv} to the case of Coulomb kernels. We recall that the sets $G_i$, $F_i$ and $S$ are defined in \eqref{defac}, \eqref{defs}.
\begin{proposition}\label{armonicofond}
Let $(f_1,f_2)$ be a minimizer of $\mathcal E ^{{c_{11}},{c_{22}}}_{K_{C_N}}$ in $\Am$.
The following facts hold true.
\begin{itemize}
\item[(i)] $({c_{11}}+ 1)f_1-({c_{22}} +1)f_2=0$ a.e. in $G_1\cap G_2$.
In particular,
if either $({c_{11}} + 1)({c_{22}} + 1)<0$, or  ${c_{11}}= -1 \ne {c_{22}}$, or  ${c_{22}}= -1 \ne {c_{11}}$, 
then  $|G_1\cap G_2|=0$.
\item[(ii)] if ${c_{11}}\neq 0$, then  $|G_1\setminus G_2|=0$, while if 
${c_{22}}\neq 0$, then $
|G_2\setminus G_1|=0$.
\item[(iii)] $G_1\cap G_2\subset S$.
\item[(iv)] If ${c_{11}}\neq -1$ or ${c_{22}}\neq -1$, then
\begin{equation}\label{zonamista}
f_1=a_1,\quad f_2=a_2 \quad\text{ a.e. in }G_1\cap G_2,
\end{equation}
where $a_i$ are defined in \eqref{ai}.
\end{itemize}
\end{proposition}

\begin{proof} 

\noindent {(i)} is a consequence of \eqref{diffconst}. To prove { (ii)} notice that $G_1\setminus G_2\subset G_1\setminus S$, which implies by 
\eqref{sommaconst} that ${c_{11}} f_1 = f_2$ in $G_1\setminus G_2$. Furthermore, in this region $f_2=0$, so that (since ${c_{11}} \neq 0$) also $f_1=0$. The case ${c_{22}}\neq 0$ is proved in the same way.

The proof of { (iii)} follows recalling that by \eqref{sommaconst} we have $0>{c_{11}} f_1-f_2=0$ in $(G_1\cap G_2)\setminus S$ and hence $|(G_1\cap G_2)\setminus S|=0$. The claim in { (iv)} follows by \eqref{diffconst} recalling that, in view of { (iii)}, $f_1+f_2=1$.
\end{proof}

\subsection{The strongly attractive case {${c_{11}} + {c_{22}} \le -2$}}
In  Theorem \ref{proptutto} we have characterized the minimizers for every ${c_{11}}, \, {c_{22}} $ such that ${c_{11}} +{c_{22}} \le -2$ and 
$\max \{{c_{11}}, {c_{22}}\} \ge -1$. Clearly such result applies also to Coulomb kernels. The (general $N$ dimensional) case ${c_{11}},{c_{22}}<-1$ was left open. In
the following proposition, we show that for Coulomb kernels the minimizers $f_i$ are characteristic functions of sets $E_i$ whose shape is unknown (see Remark \ref{cong} for some further comments in this direction).

\begin{proposition}\label{coroll1.5}
Let ${c_{11}}+{c_{22}}\le-2$ with $({c_{11}},{c_{22}})\neq (-1,-1)$. If $(f_1,f_2)$ is a minimizer of $\mathcal E^{{c_{11}},{c_{22}}}_{K_{C_N}}$ in $\Am$, then $f_1=\chi_{F_1}$ and $f_2=\chi_{F_2}$ for some $F_1$, $F_2\subset \R^N$.
\end{proposition}

\begin{proof}
By Theorem \ref{proptutto} and Proposition \ref{oned} the claim holds true in the one dimensional case and in the general $N$ dimensional case for $\max\{{c_{11}},{c_{22}}\}\ge -1$, so that  it is enough to prove the claim in the case $N\ge 2$ and ${c_{11}},{c_{22}}<-1$. Since  ${c_{11}}+{c_{22}}+2<0$, by applying Lemma \ref{secondvar} with $\ffi  \in L^1(\R^N;\R)$,   $\ffi=0$ a.e. in $\R^N\setminus (G_1\cap G_2)$ and $\int_{\R^N}\ffi  \ud x =0$, we get
\begin{equation}\label{minoridimenouno}
\int_{G_1\cap G_2}\int_{G_1\cap G_2}  K_{C_N}(x-y)  \, \ffi(x) \, 
\ffi(y) \ud x\ud y\le 0.
\end{equation}
By Remark \ref{posdef} 
we deduce that the above inequality is actually an equality and that  $\ffi=0$  in $G_1\cap G_2$. By the arbitrariness of $\ffi$, it follows that $|G_1\cap G_2|=0$.
Finally, by Proposition \ref{armonicofond}(ii), we have that $|G_1\setminus G_2|=|G_2\setminus G_1|=0$, so we conclude that
$|G_1|=|G_2|=0$.

\end{proof}

\subsection{The weakly attractive case {${c_{11}} +  {c_{22}} > -2$}  (preliminary results)}

For any measurable set $E\subset\R^N$, we set $V_E:=\chi_{E}* K$.
\begin{lemma}\label{solounaomogenea}
Let $-1\le{c_{11}},{c_{22}}\le 0$ with ${c_{11}}\neq -1$ or ${c_{22}}\neq -1$. Then,
there exists a minimizer $(f_1,f_2)$ of $\mathcal E^{{c_{11}},{c_{22}}}_{K_{C_N}}$ in $\Am$, such that $|G_1\setminus G_2|=|G_2\setminus G_1|=0$
and either $|F_1|=0$ or $|F_2|=0$.

Moreover, any minimizer $(f_1,f_2)$ of $\mathcal E^{{c_{11}},{c_{22}}}_{K_{C_N}}$ in $\Am$ is such that either $|G_1\setminus G_2| + |F_1|=0$ or
$|G_2\setminus G_1| + |F_2|=0$. 
\end{lemma}

\begin{proof}
Let $(f_1,f_2)$ be a minimizer of $\mathcal E^{{c_{11}},{c_{22}}}_{K_{C_N}}$ in $\Am$.
By Proposition \ref{armonicofond}(ii) and Corollary \ref{concentra} we can always assume
\begin{equation}\label{mult}
\{f_1\neq 0\}=(G_1\cap G_2)\cup F_1,\qquad \{f_2\neq 0\}=(G_1\cap G_2)\cup F_2\quad\text{a.e.},
\end{equation}
so that $|G_1\setminus G_2|=|G_2\setminus G_1|=0$. 

Now, let us prove that either $|F_1|=0$ or $|F_2|=0$. We first focus on the case $N\ge 3$. By \eqref{mult} and \eqref{zonamista} we have
\begin{equation}\label{zonamista2}
f_1=\textstyle a_1 \chi_{G_1\cap G_2}+\chi_{F_1},\quad f_2=a_2 \chi_{G_1\cap G_2}+\chi_{F_2}.
\end{equation}
It follows that 
$$
V_1=\textstyle a_1 V_{G_1\cap G_2}+V_{F_1},\quad V_2=a_2 V_{G_1\cap G_2}+V_{F_2}
$$
which together with \eqref{ineq} easily yields 
\begin{equation}\label{ineq1}
({c_{11}}+1)V_{F_1}(x_2)-({c_{22}}+1)V_{F_2}(x_2)\ge({c_{11}}+1)V_{F_1}(x_1)-({c_{22}}+1)V_{F_2}(x_1)
\end{equation}
for any $x_1\in \bar{F_1}$ and any $x_2\in \bar{F_2}$. 
Set $U(x):=({c_{11}}+1)V_{F_1}(x)-({c_{22}}+1)V_{F_2}(x)$. Then $U$ solves
\begin{equation}\label{equation}
\left\{\begin{array}{ll}
-\Delta U=({c_{11}}+1)\chi_{F_1}-({c_{22}}+1)\chi_{F_2}&\text{ in }\R^N\\
U(x)\to 0&\text{ if }|x|\to\infty.
\end{array}\right.
\end{equation}
So, $U$ is subharmonic  in $\R^N\setminus \overline{F}_1$ and hence either $U\le 0$ or $U$ reaches its maximum on $\overline{F}_1$. Analogously, since $U$ is superharmonic in  $\R^N\setminus \overline{F}_2$,  either $U\ge 0$ or $U$ reaches its minimum on $\overline{F}_2$. Now, if $U\equiv 0$, then $|F_1|=|F_2|=0$; otherwise, assume, for instance, that $U$ reaches its maximum on $\overline{F}_1$.
By \eqref{ineq1} and by \eqref{equation}, it follows that $U$ is constant in $F_2$,  and hence $|F_2|=0$.
Analogously, if $U$ reaches its minimum on $\overline F_2$, we get that $|F_1|=0$.

The proofs for the cases $N=1,2$ are analogous, the only care being that, for $N=2$, the boundary condition in  \eqref{equation} should be replaced either by $U(x)\to 0$ or $U(x)\to\pm \infty$, according with the sign of $({c_{11}}+1)|F_1|-({c_{22}}+1)|F_2|$. For $N=1$ a direct proof shows that  $U$ reaches its maximum on $\overline{F}_1$ and its minimum on $\overline{F_2}$.

We pass to the proof of the last claim of the lemma. Assume by contradiction that $G_1\setminus G_2 + |F_1|>0$ and
$|G_2\setminus G_1| + |F_2|>0$. By Proposition \ref{armonicofond}(ii) and Corollary \ref{concentra} we deduce that there exists a minimizer satisfying \eqref{mult} with 
both $F_1$ and $F_2$  with positive measure.  Following the lines of the proof of the first claim of the lemma, this provides a contradiction. 
\end{proof}

The remaining part of this section is devoted to the uniqueness and characterization of the minimizer. In particular, we will see that the unique minimizer in the purely weakly attractive case, corresponding to $-1<{c_{11}},{c_{22}}\le 0$, is given by a ball where the two phases are mixed proportionally to their self attraction coefficents and by an annulus around this ball (see Corollary \ref{quadr1d} for the case $N=1$ and Theorem \ref{00} and Corollary \ref{minimoinquadrato} for the case $N\ge 2$). Moreover,  we will see that also in the reamainig cases, i.e., ${c_{11}}\le-1\le c_{22}\le 0$ and ${c_{22}}\le-1\le c_{11}\le 0$, with $ {c_{11}} + {c_{22}}>-2$, the unique minimizer is given once again by a ball and an annulus around it, where the internal ball corresponds to the phase having the stronger self-attraction coefficient (see Proposition \ref{triangolino1d} for the case $N=$1 and Corollary \ref{triang} for the case $N\ge 2$). 


\subsection{The weakly attractive case {${c_{11}} +  {c_{22}} > -2$} (in dimension $N=1$)}
In the following proposition we study the minimizer of $\mathcal E^{{c_{11}},{c_{22}}}_{K_{C_1}}$ when ${c_{11}}\le-1\le c_{22}\le 0$ and ${c_{11}}+{c_{22}}>-2$. In the subsequent corollary we take advantage of this result via a re-parameterization of the energies to study the case $-1<{c_{11}},{c_{22}}\le 0$.

\begin{proposition}\label{triangolino1d}
Let ${c_{11}}\le -1$ and $-1\le c_{22}\le 0$ (resp. ${c_{22}}\le -1$ and $-1\le {c_{11}}\le 0 $)  with ${c_{11}}+{c_{22}}>-2$. Then 
the (unique up to a translation) minimizer of $\E_{K_{C_1}}^{c_{11},c_{22}}$ in $\Am$ is given by the pair
$$
(f_1,f_2)=(\chi_{B^{m_1}},\chi_{B^{m_1+m_2}\setminus B^{m_1}})\quad\text{(resp. }(f_1,f_2)=(\chi_{B^{m_1+m_2}\setminus B^{m_2}},\chi_{B^{m_2}})).
$$
\end{proposition}

\begin{proof}
We prove the claim only for ${c_{11}}\le-1$ and $-1\le c_{22} \le 0$ with ${c_{11}}+{c_{22}}> -2$, being the proof of the other case analogous.
Let $(f_1,f_2)$ be a minimizer of $\E_{K_{C_1}}^{c_{11},c_{22}}$ in $\Am$. By  (i) and (ii) of Proposition \ref{armonicofond}, we have that $f_1=\chi_{F_1}$
 and $f_2=\chi_{F_2}+f_2\res{G_2}$. We can assume without loss of generality that $F_1\cup F_2\cup G_2$ is an interval, 
 since reducing the distances decreases the energy. For the same reason,  it is easy  to see that $|G_2|=0$.
Notice that
$$
\mathcal E^{{c_{11}},{c_{22}}}_{K_{C_1}}(f_1,f_2)=\mathcal E^{-1,{c_{22}}}_{K_{C_1}}(f_1,f_2) + ({c_{11}}+1)J_{K_{C_1}}(f_1,f_1),
$$
so it is enough to prove the claim for ${c_{11}}=-1$.
We now prove that $V_2'=0$ in $F_1$. By \eqref{ineq}, we have
$$
V_2(x_1)\ge V_2(x_2)\quad\text{for any }x_1\in \overline{F_1}\text{ and }x_2\in \overline{ F_2},
$$
 and, by the  maximum principle,  $V_2$ attains its maximum in $\overline{F_2}$ (notice that $V_2\to -\infty$ as $|x|\to +\infty$). It follows that  for any $x\in F_1$, $V_2(x)=\max V_2$.
We have
$$
0=V_2'(x)=\frac 12\left(|F_2\cap(-\infty, x]|-|F_2\cap[x,\infty)|\right) \quad \text{ for any $x\in F_1$,}
$$
and hence $F_1$ is connected and centered  in  $F_1\cup F_2$. 
\end{proof}

\begin{corollary}\label{quadr1d}
Let $-1< {c_{11}},{c_{22}}\le 0$. 
Then,  the  following results hold true (recall that $a_i$ are defined in \eqref{ai}).
\begin{itemize}
\item[(i)] If $({c_{22}}+1)m_2\ge ({c_{11}}+1)m_1$, then (up to a translation)
\begin{equation}\label{minimo1}
(f_1,f_2)=\Big(a_1\,\chi_{B^{\frac{m_1}{a_1}}},\chi_{B^{m_2+m_1}}-a_1\,\chi_{B^{\frac{m_1}{a_1}}}\Big)
\end{equation}
is the (unique) minimizer of $\mathcal E^{{c_{11}},{c_{22}}}_{K_{a_2}}$ in $\Am$.
\item[(ii)] If $({c_{11}}+1)m_1> ({c_{22}}+1)m_2$, then (up to a translation)
\begin{equation*}
(f_1,f_2)=\Big(\chi_{B^{m_2+m_1}}-a_2\,\chi_{B^{{\frac{m_2}{a_2}}}},a_2\,\chi_{B^{{\frac{m_2}{a_2}}}}\Big)
\end{equation*}
is the (unique) minimizer of $\mathcal E^{{c_{11}},{c_{22}}}_{K_{a_2}}$ in $\Am$.
\end{itemize}
\end{corollary}

\begin{proof}
We prove only (i) since the proof of (ii) is analogous.

Let $(f_1,f_2)$ be a minimizer of $\E_{K_{C_1}}^{c_{11},c_{22}}$ in $\Am$.
Arguing as in the proof of Proposition \ref{triangolino1d}, one can show that $|G_2\setminus G_1| + |G_1\setminus G_2|=0$, and hence
\begin{equation}\label{casoi}
f_{1}=a_1\chi_{G_1\cap G_2}\text{ and }f_{2}=a_2\chi_{G_1\cap G_2}+\chi_{F_2}.
\end{equation}
Set $A:=G_1= G_2$, $B:=F_2$, $\tilde m_1:={\frac{m_1}{a_1}}$ and $\tilde m_2:=m_2-\frac{{c_{11}}+1}{{c_{22}}+1}m_1>\tilde m_1$; then, by easy computations, it follows that 
\begin{eqnarray*}
\textstyle \mathcal E^{{c_{11}},{c_{22}}}_{K_{C_1}}(f_1,f_2)&=&\frac{1-{c_{11}}{c_{22}}}{{c_{11}}+{c_{22}}+2}[-J_{K_{C_1}}(A,A)+{c_{22}}\frac{{c_{11}}+{c_{22}}+2}{1-{c_{11}}{c_{22}}}J_{K_{C_1}}(B,B)\\
&&-2 J_{K_{C_1}}(A,B)]\\ 
&=&\frac{1-{c_{11}}{c_{22}}}{{c_{11}}+{c_{22}}+2}\E_{K_{C_1}}^{-1,\tilde c_{22}}(\chi_A,\chi_B),
\end{eqnarray*}
with $\tilde c_{22}:={c_{22}}\,\frac{{c_{11}}+{c_{22}}+2}{1-{c_{11}}{c_{22}}}\in(-1,0)$. Since $\frac{1-{c_{11}}{c_{22}}}{{c_{11}}+{c_{22}}+2}>0$, it follows that $(f_1,f_2)$ is a minimizer of $\mathcal E^{{c_{11}},{c_{22}}}_{K_{C_1}}$ in $\Am$ if and only if $(\chi_A,\chi_B)$ minimizes $\mathcal E_{K_{C_1}}^{-1,\tilde c_{22}}$ in $ \mathcal A_{\tilde m_1,\tilde m_2}$. By Proposition \ref{triangolino1d}, the unique minimizer of $\mathcal{E}_{K_{C_1}}^{-1,\tilde c_{22}}(\chi_A,\chi_B)$ among the pairs $(A,B)$ with $|A|=\tilde m_1$ and $|B|=\tilde m_2$ is given by $(B^{\tilde m_1},B^{\tilde m_1+\tilde m_2}\setminus B^{\tilde m_1})$. The claim follows thanks to formula \eqref{casoi}.
\end{proof}

One might wonder whether the assumption that $K=K_{C_1}$ is crucial in order to prove Proposition \ref{triangolino1d} and Corollary \ref{quadr1d}. In the following Remark, we exhibit an example of kernel for which the pair $(f_1,f_2)$ in \eqref{minimo1} is not the minimizer of $\E_K^{0,0}$ in $\A_{m_1,m_2}$, for suitably chosen $m_1,m_2>0$.

\begin{remark}\label{controes}
\rm{ 
Let $\rho>0$ and let $m_1,m_2>0$ be such that $m_1>2\rho$, $m_2>m_1+4\rho$. Consider the kernel $K:=\chi_{[-\rho,\rho]}$ and set $A:=(-m_1,m_1)$, $B:=(-\frac{m_1+m_2}{2},\frac{m_1+m_2}{2})$,  $(f_1,f_2)=(\frac 1 2 \chi_{A},\chi_{B}-\frac 1 2 \chi_{A})$. Then, 
$$
\textstyle \E^{0,0}(f_1,f_2)=-[-\frac 1 2 J_{K}(A,A)+J_K(A,B)].
$$
One can easily check that
$J_{K}(A,A)=4\rho m_1-\rho^2$ and $J_{K}(A,B)=4\rho m_1$; it follows that
$$
\E_K^{0,0}(f_1,f_2)=-\rho(2m_1+\textstyle \frac\rho 2).
$$
Let now split $A$ into two intervals $A_1:=(-{c_{11}}\rho-m_1,-{c_{11}}\rho)$ and $A_2:=({c_{11}}\rho,{c_{11}}\rho+m_1)$, with $\frac 1 2 <{c_{11}}<1$, and consider the energy of the admissible pair
$$
(g_1,g_2):=(\textstyle \frac 1 2 \chi_{A_1}+\frac 1 2 \chi_{A_2},\chi_{B}-\frac 1 2 \chi_{A_1}-\frac 1 2 \chi_{A_2}).
$$
By symmetry $J_{K}(A_2,A_2)= J_K(A_1,A_1)$ and $  J_K(A_2,B)=J_K(A_1,B)$. Hence
$$
\E_K^{0,0}(g_1,g_2)=-[\textstyle - J_K(A_1,A_1)- J_K(A_1,A_2)+ 2 J_K(A_1,B)],
$$
where  $J_K(A_1,A_1)=2\rho m_1-\rho^2$, $J_{K}(A_1,A_2)=0$ (since ${c_{11}}>\frac 1 2$) and  $J_K(A_1,B)=2\rho m_1$.
It follows that $\E_K^{0,0}(g_1,g_2)=-\rho(2m_1+\rho)<\E_K^{0,0}(f_1,f_2)$ and therefore $(f_1,f_2)$ is not the minimizer of $\E_K^{0,0}$ in $\Am$.
One can easily check that the above result holds true also taking $K(x):=\chi_{[-\rho,\rho]}(x)\,(\rho-|x|)$ and $m_1,m_2$ as above.

}
\end{remark}
\subsection{The weakly attractive case {${c_{11}} +  {c_{22}} > -2$} 
(the case $N\ge 2$)}
Now we focus on the case $N\ge 2$, considering first the case ${c_{11}}={c_{22}}=0$ (Theorem \ref{00}) and then the remaining cases (Corollaries \ref{triang} and \ref{minimoinquadrato}).

We first introduce some preliminary notation and recall some well known results we will use in this section.
For any $g\in L^{2}(\R^N;\R^+)$, we  set $V:=K_{C_N}\ast g$. Moreover, we recall that for every function $u\in L^1(\R^N;\R^+)$,  $u^*$ is the spherical symmetric nonincreasing rearrangement of $u$ defined in \eqref{simmetrizz}. Clearly,  the notion of spherical symmetric nonincreasing rearrangement can be extended in the obvious way to functions  $u\in L_{loc}^1(\R^N;\R)$ tending to $-\infty$ for $x\to +\infty$.
\begin{lemma}\label{asi}
Let $g\in L^{2}(\R^N;\R^+)$, let $ m:= \int_{\R^N} g(x) \ud x$, and let $V:= K_{C_N}*g$. Moreover, for $N=2$ assume that $g$ has compact support. Then,
\begin{align}
& V(x)\to 0 \text{ as } |x|\to + \infty & \qquad \text{ for } N\ge3; \\
& V(x)  = -\frac{m}{2\pi} \log|x| + r(x) & \qquad \text{ for } N=2;
\end{align}
where $r(x)\to 0$ as $|x|\to +\infty$. As a consequence, $V-V^* \to 0 $ as $|x|\to +\infty$.
 \end{lemma}
Let now $f\in L^1(\R^N;\R^+)$. For any $r>0$ we denote by $t(r)$ the unique $t\in\R$ such that $|\{V>t\}|\le \omega_N r^N\le |\{V\ge t\}|$. Let  $\tilde f:\R^N\to \R$ be defined by
\begin{equation}\label{efti}
\textstyle \tilde f(x):=\frac{1}{N\omega_N |x|^{N-1}}\frac{\ud t}{\ud r}_{\big |_{t=t(|x|)}}\frac{\ud}{\ud t}\left( \int_{\{V>t\}}f(y)\ud y\right)_{\big |_{t=t(|x|)}}.
\end{equation}
We notice that $B_r=\{V^*>t(r)\}$ and that
\begin{equation}\label{nosce0}
\int_{\{V^*>t\}}  \tilde f(x) \ud x =  \int_{ \{V >t \}} f(x) \ud x \qquad \text{ for every } t\in\R.
\end{equation}
Moreover, one can easily check that also $\tilde f$ takes values in $\R^+$, and 
\begin{equation}\label{nosce}
\|\tilde f\|_{1}=\|f\|_1,\qquad  \|\tilde f\|_{p}\le \|f\|_p \quad \text{for all } 1< p \le +\infty. 
\end{equation}

\begin{lemma}\label{rearr0}
 Let $f\in L^{2}(\R^N;\R^+)$, with $N\ge 2$,  and let $V:=K_{C_N}\ast f$.  Moreover,  let $\tilde f \in L^{2}(\R^N;\R^+)$  be defined as in \eqref{efti}, and  
  let $\tilde V := K_{C_N}\ast \tilde f$.
Then, $\tilde V \ge V^*$, and 
$$
\tilde V(x) > V^*(x) \qquad \text{ for a.e. } x\in B_{r(t_{max})},
$$
where $t_{max}$ is the maximal level such that $\{V>t\}$ is a ball for every $ t\le t_{max}$. 
\end{lemma}

%

\begin{proof}
By the coarea formula and the isoperimetric inequality, for almost every $t\in\R$ we have
$$
\int_{\partial \{V>t\}} |\nabla V(x)| \, \ud\hs^{N-1} \ge \int_{\partial \{V^*> t\}} |\nabla V^*(x)| \, \ud\hs^{N-1},
$$
with strict inequality whenever $\{V>t\}$ is not a ball.
Therefore, by \eqref{nosce0}
\begin{multline}
\int_{\partial \{V^*> t\}} |\nabla \tilde V(x)| \, \ud\hs^{N-1} \geq   - \int_{\{V^* >t\}} \Delta  \tilde V(x)  \ud x
= - \int_{\{V >t\}} \Delta   V(x)  \ud x 
\\
= \int_{\partial \{V>t\}} |\nabla V(x)| \, \ud\hs^{N-1} \ge \int_{\partial \{V^*> t\}} |\nabla V^*(x)| \, \ud\hs^{N-1}
\end{multline}
with strict inequalities whenever $\{V>t\}$ is not a ball.
Since $\tilde V - V^*$ is radial and  
 in view of Lemma \ref{asi}  it  vanishes at infinity, we have
\begin{multline}
\tilde V(r) - V^*(r) = \int_{r}^{+\infty} \frac{\ud}{\ud s} \left(V^*(s) - \tilde V(s)\right) \ud s =\\ 
\int_{r}^{+\infty} \frac{1}{N \omega_N s^{N-1}}  \ud s \int_{\partial B_s} - |\nabla V^*(x)| + |\nabla \tilde V(x)|\ud x .
\end{multline}
The claim follows since the integrand is nonnegative, and it is strictly positive in a subset of positive measure of $(r,+\infty)$, for all  $r <  r_{t_{max}}$.


\end{proof}
Lemma \ref{rearr0} establishes that we can rearrange the mass of $f_1$ in order to obtain a new radial charge configuration $\tilde f_1$, increasing the corresponding potential. Exploiting such a result, we deduce that the minimizer of  $\E_K^{0,0}$ has radial symmetry. This is done in the next theorem.

\begin{theorem}\label{00}
For $m_2\ge m_1$,  the (unique up to a translation) minimizer of $\E_{K_{C_N}}^{0,0}$ in $\Am$ is given by the pair $(f_1,f_2)$ where
$$
\textstyle f_1:= \frac12 \chi_{B^{2m_1}}, \qquad  f_2:= \chi_{B^{m_1+m_2}}-\frac12 \chi_{B^{2m_1}}.
$$
\end{theorem}

\begin{proof}
Let $(f_1,f_2)$ be a minimizer of $\mathcal{E}_{K_{C_N}}^{0,0}$ in $\Am$. Let $V_1$ be the potential generated by $f_1$ and  let $\tilde f_1$ and $\tilde V_1$ be defined according to Lemma \ref{rearr0}. Notice that $0\le \tilde f_1 \le 1$, and that $\|\tilde f_1\|_{L^1(\R^N)} = m_1$. Let us observe that by standard regularity theory, $\tilde V_{1}$ attains a maximum. We denote it by $\tilde M_{1}$. We first show that there exists $\tilde t<\tilde M_{1}$ such that
\begin{equation}\label{ovvia}
\int_{\{\tilde V_1>\tilde t\}} (1-\tilde f_1(x)) \ud x = m_2.
\end{equation}
Suppose by contradiction that there does not exist $\tilde t$ such that \eqref{ovvia} holds true. Notice that $-\Delta \tilde V = \tilde f$ and that $\tilde f$ and $\tilde V$  are radially symmetric. Therefore, $\tilde V$ may have a flat region only in a ball centered at the origin, whereas it is strictly decreasing with respect to $|x|$ elsewhere. We deduce that   $\int_{\{\tilde V_1> t\}} (1-\tilde f_1(x)) \ud x > m_2$ for any $t< \tilde M_{1}$, and  in particular that $\int_{\{\tilde V_1=\tilde M_{1}\}} (1-\tilde f_1(x)) \ud x \ge m_2$. It follows that $|\{\tilde V_{1}=\tilde M_{1}\}|\ge m_{2}$ and, since $\tilde V_{1}$ is radially symmetric, $\{\tilde V_{1}=\tilde M_{1}\}$ is a ball centered at the origin containing  $B^{m_2}$. Set $\hat f_{2}:=\chi_{B^{m_{2}}}$; we have that $(\tilde f_{1},\hat f_{2})\in\Am$, and by Lemma \ref{corsuper}
$(\tilde f_{1},\hat f_{2})$ is a minimizer. 
Notice that  
the supports of $\tilde f_1$ and $\hat f_2$ are disjoint, but this is in contradiction with Proposition \ref{solounaomogenea}.
We conclude  that  there exists $\tilde t$ satisfying \eqref{ovvia}. 
Set 
$$
\tilde f_2(x) := 
\begin{cases}
1- \tilde f_1(x) & \text{ for } x\in \{\tilde V_1>\tilde t\}
\\
0 & \text{ otherwise.}
\end{cases};
$$   
by construction $(\tilde f_1,\tilde f_2)\in  \A_{m_1,m_2}$ ($\int_{\R^N}\tilde f_2(x)\ud x=m_2$ by \eqref{ovvia}).

Let now $\hat t\le \tilde t$ be such that 
$$
\{ V^*_1>\hat t\}\subseteq \{\tilde V_1>\tilde t\}
\subseteq \{ V^*_1\ge \hat t\}\,. 
$$
This is possible since the superlevel set
 $\{\tilde V_1>\tilde t\}$ is a ball centered at the origin.
Let $A:=\{\tilde V_1>\tilde t\}\setminus \{ V^*_1>\hat t\}$. Since 
$A\subseteq \{V_1^* = \hat t \}$, we have 
$\tilde f_1=0$ a.e. on $A$, and hence 
$\tilde f_2=1$ a.e. on $A$. Moreover, by Corollary \ref{concentra} we can always assume that 
\begin{equation}\label{Ap}
\text{ supp } f_2=  \{ V_1\ge \hat t\} \cup A', \quad f_2 = 1-f_1 \text{ on } \{V_1>\hat t\},  \quad  f_2\equiv 1 \text{ on } A',
\end{equation}
for some  set $A' =   \{ V_1\ge \hat t\}$ with 
$|A'|=|A|$.
By the Coarea Formula and Lemma \ref{rearr0} we  have 
\begin{eqnarray}\nonumber
\E_{K_{C_N}}^{0,0} (\tilde f_1,\tilde f_2) 
&=& -2 \int_{\R^N} \tilde f_2(x) \tilde V_1(x)\ud x\\ \label{primadisug}
&\le&  -2 \int_{\R^N} \tilde f_2(x)  V^*_1(x)\ud x \\ \nonumber
&=&-2\,\hat t\, |A|
-2\int_{\hat t}^{+\infty} t\int_{\{V_1^*>t\}}(1-\tilde f_1(x))\ud x\ud t
\\ \nonumber
&=&-2\,\hat t\,|A|-2\int_{\hat t}^{+\infty} t\int_{\{V_1>t\}}(1- f_1(x))\ud x\ud t
\\ \nonumber
&=&-2\,\hat t\,|A|-2\int_{\{V_1>\hat t\}}(1-f_1(x))V_1(x)\ud x\\
\label{spec}
&= & -2\int_{A'} f_2(x) \, \hat t \ud x -2\int_{\{V_1>\hat t\}}(1-f_1(x))V_1(x)\ud x
\\
\label{esp}
&=& -2\int_{\R^N} f_2(x)V_1(x)\ud x
=\E_{K_{C_N}}^{0,0} (f_1,f_2),
\end{eqnarray}
where the  equality in \eqref{spec} follows from \eqref{Ap}. 
By minimality, the inequality in   \eqref{primadisug} is actually an equality, and hence  $V_1^*\equiv \tilde V_1$. It follows that 
all the superlevels of $V_1$ are balls. 
By Proposition \ref{armonicofond}, $f_1 = \frac 12$ in $G_1\cap G_2$, whereas by Lemma \ref{solounaomogenea} $G_1\cup F_1 = G_1\cap G_2$, so that  
$f_1:= \frac12 \chi_{E}$ for some set $E$. Since all the 
superlevel sets of $V_1$ are balls, we conclude that, up to a translation, $f_1:= \frac12 \chi_{B^{2m_1}}$. By Corollary \ref{corsuper} we also deduce that  $f_2:= \chi_{B^{m_1+m_2}}-\frac12 \chi_{B^{2m_1}}$ and this concludes the proof.

\end{proof}


\begin{corollary}\label{triang}
Let ${c_{11}}\le -1$ and $-1	\le c_{22}\le  0$ (resp. ${c_{22}}\le -1$ and $-1\le {c_{11}}\le 0$)  with ${c_{11}}+{c_{22}}>-2$. Then, the  
 (unique up to a translation) minimizer of $\mathcal E^{{c_{11}},{c_{22}}}_{K_{C_N}}$ in $\Am$ is given by the pair 
$$
(f_1,f_2)=(\chi_{B^{m_1}},\chi_{B^{m_1+m_2}\setminus {B^{m_1}}})\quad\text{(resp. }(f_1,f_2)=(\chi_{B^{m_1+m_2}\setminus{B^{m_2}}},\chi_{B^{m_2}}).
$$

\end{corollary}

\vskip -0.2cm
\begin{figure}[h]
\centering
{\def\svgwidth{250pt}
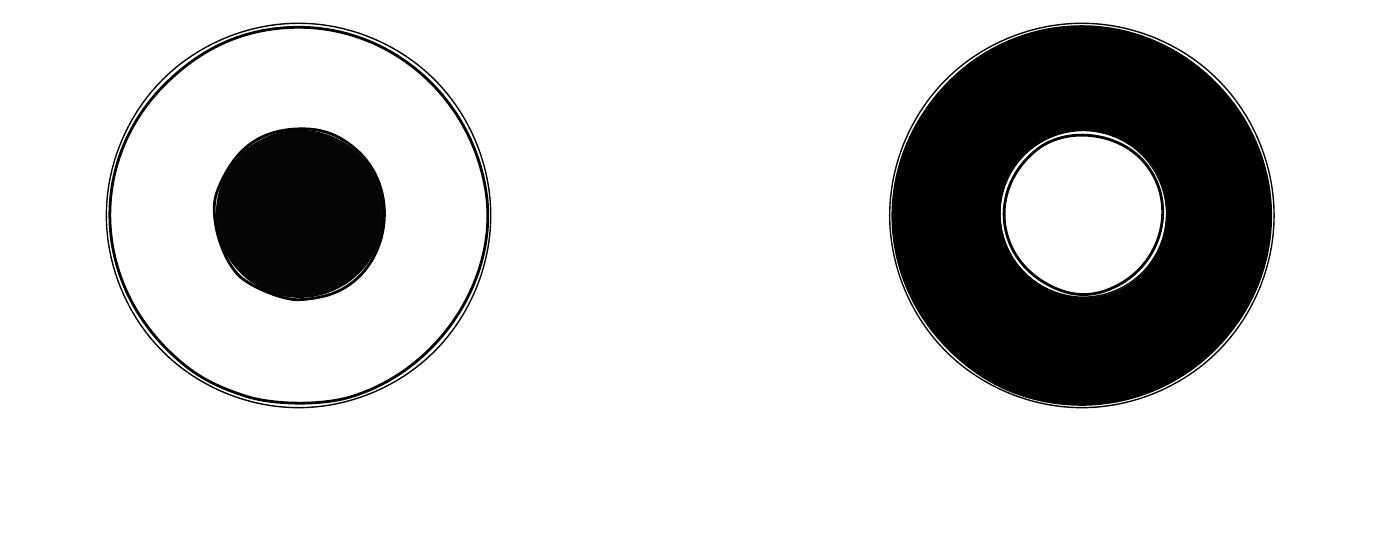}
\vskip -0.7cm
\caption{The phase $f_1$ is black and the phase $f_2$ is white. The minimizer in the case ${c_{11}}\le -1\le c_{22}\le  0$ is represented on the left, whereas on the right there is the minimizer in the case ${c_{22}}\le -1\le {c_{11}}\le 0$.}
\end{figure}

\begin{proof}
We prove the claim only for ${c_{11}}\le-1$ and $-1\le c_{22}\le  0$ with ${c_{11}}+{c_{22}}>-2$,  the proof of the other case being fully analogous.
Let $(f_1,f_2)$ be a minimizer of $\E_{K_{C_N}}^{c_{11},c_{22}}$ in $\Am$. 

Set $\tilde m_1:=\frac{m_1}{2}$, $\tilde m_2:=\frac{m_1}{2}+m_2>\tilde m_1$, 
\begin{equation}\label{sostituz}
g_1:=\frac{f_1}{2},\qquad g_2:=\frac{f_1}{2}+f_2.
\end{equation}
It is easy to see that $g_i\ge 0$,  $\int_{\R^N}g_i(x)\ud x=\tilde m_i$ (for $i=1,2$) and $g_1+g_2=f_1+f_2\le 1$, so that $(g_1,g_2)\in \mathcal A_{\tilde m_1,\tilde m_2}$. A straightforward computation yields 
\begin{eqnarray}\nonumber
\mathcal E^{{c_{11}},{c_{22}}}_{K_{C_N}}(f_1,f_2)&=&({c_{11}}+1) J_{K_{C_N}}(f_1,f_1)+\mathcal E_{K_{C_N}}^{-1,{c_{22}}}(f_1,f_2)\\ 
\nonumber
&=&({c_{11}}+1) J_{K_{C_N}}(f_1,f_1)
+{c_{22}} J_{K_{C_N}}(f_1+f_2,f_1+f_2)
\\ \nonumber
&&+(1+{c_{22}})(-J_{K_{C_N}}(f_1,f_1)-2J_{K_{C_N}}(f_1,f_2))\\ \label{ultimofunz} 
&=&4\,({c_{11}}+1)J_{K_{C_N}}(g_1,g_1)+{c_{22}} J_{K_{C_N}}(g_1+g_2,g_1+g_2)\\ \nonumber
&&+2(1+{c_{22}})\,
\mathcal{E}_{K_{C_N}}^{0,0}(\textstyle g_1,g_2),
\end{eqnarray}
and hence $(f_1,f_2)$ is a minimizer of $\mathcal E_{K_{C_N}}^{{c_{11}},c_{22}}$ in $\Am$ if and only if $(g_1,g_2)$  minimizes the 
energy 
\begin{equation}\label{newen}
4\,({c_{11}}+1)J_{K_{C_N}}(g_1,g_1)+{c_{22}} J_{K_{C_N}}(g_1+g_2,g_1+g_2)+2(1+{c_{22}})\,
\mathcal{E}_{K_{C_N}}^{0,0}(\textstyle g_1,g_2)
\end{equation}
in $ \mathcal A_{\tilde m_1,\tilde m_2}$.
By Theorem \ref{00}, 
 the third addendum in \eqref{newen} is minimized (in $\mathcal A_{\tilde m_1,\tilde m_2}$) if and only if 
 $$
 (g_1,g_2)=(\textstyle \frac 1 2\chi_{B^{2\tilde m_1}},\frac 1 2 \chi_{B^{2\tilde m_1}}+\chi_{B^{\tilde m_1+\tilde m_2}\setminus B^{2\tilde m_1}}).
 $$
We notice that such configuration minimizes also the first and the second addendum. The claim follows directly by \eqref{sostituz}.
 \end{proof}
 
Quantitative Riesz inequalities have been recently studied in  \cite[Theorem 1.5]{CM}.  For any measurable set $E\subset \R^N$ with finite measure, let $E^*:= B^{|E|}$ be  the ball centered at the origin such that 
$|E^*|=|E|$.
 From Corollary \ref{triang} with ${c_{11}}=-1$ and ${c_{22}}=0$
we immediately get the following improved Riesz inequality.

\begin{corollary}\label{improvedRiesz}
For any measurable sets $E_1\subseteq E_2\subset \R^N$  with finite measure, there holds
\begin{equation}\label{impri1}
J_{K_{C_N}}(E_1^*,E_2^*)-J_{K_{C_N}}(E_1,E_2)\ge \frac 12\left( 
J_{K_{C_N}}(E_1^*,E_1^*)-J_{K_{C_N}}(E_1,E_1)\right).
\end{equation}
Moreover, for any measurable sets $A_1\subseteq A_2\subset \R^N$  with finite measure, there holds
\begin{multline}\label{impri2}
J_{K_{C_N}}(A_2,A_2) - J_{K_{C_N}}(A_1,A_1) \le J_{K_{C_N}}(B^{|A_2|},B^{|A_2|}) 
\\
- J_{K_{C_N}}(B^{|A_2|}\setminus B^{|A_2| - |A_1| },B^{|A_2|}\setminus  B^{|A_2| - |A_1|}). 
\end{multline}

\end{corollary}
\begin{proof}
We only prove \eqref{impri1}, since \eqref{impri2} is indeed equivalent to \eqref{impri1} replacing $E_1$ with $A_2\setminus A_1$ and $E_2$ with $A_2$. 

Let $f_1:= \chi_{E_1}$, $f_2:=\chi_{E_2\setminus E_1}$. By Corollary \ref{triang} we have
\begin{multline*}
J_{K_{C_N}}(E_1,E_1) -2 J_{K_{C_N}}(E_1,E_2) = J_{K_{C_N}}(f_1,f_1) - 2J_{K_{C_N}}(f_1,f_1+f_2) 
\\
= \E^{-1,0}_{K_{C_N}}(f_1,f_2)
\ge \E^{-1,0}_{K_{C_N}}(\chi_{E_1^*},\chi_{E_2^*\setminus E_1^*}) = 
J_{K_{C_N}}(E_1^*,E_1^*)  -2 J_{K_{C_N}}(E_1^*,E_2^*).
\end{multline*}
\end{proof}
In the next corollary we will consider the case $-1< {c_{11}}\le 0$, $-1<{c_{22}}\le 0$, completing the analysis of the weakly attractive case for the Coulomb interaction kernel. 
Recall  the coefficients $a_i$  defined in \eqref{ai}.
\begin{corollary}\label{minimoinquadrato}
Let $-1< {c_{11}}\le 0$, $-1<{c_{22}}\le 0$.
The following results hold true.
\begin{itemize}
\item[(i)] If $({c_{22}}+1)m_2> ({c_{11}}+1)m_1$, then 
 the (unique up to a translation) minimizer of $\E^{{c_{11}},{c_{22}}}_{K_{C_N}}$ in $\Am$ is given by the pair 
\begin{equation*}
(f_1,f_2)=\Big(a_1\,\chi_{B^{\frac{m_1}{a_1}}},\chi_{B^{m_2+m_1}}-a_1\,\chi_{B^{\frac{m_1}{a_1}}}\Big).
\end{equation*}
\item[(ii)] If $({c_{11}}+1)m_1> ({c_{22}}+1)m_2$, then 
 the (unique up to a translation) minimizer of $\E^{{c_{11}},{c_{22}}}_{K_{C_N}}$ in $\Am$ is given by the pair 
\begin{equation*}
(f_1,f_2)=\Big(\chi_{B^{m_2+m_1}}-a_2\,\chi_{B^{\frac{m_2}{a_2}}},a_2\,\chi_{B^{\frac{m_2}{a_2}}}\Big).
\end{equation*}
\end{itemize}
\end{corollary}

\begin{figure}[h]
\centering
{\def\svgwidth{250pt}
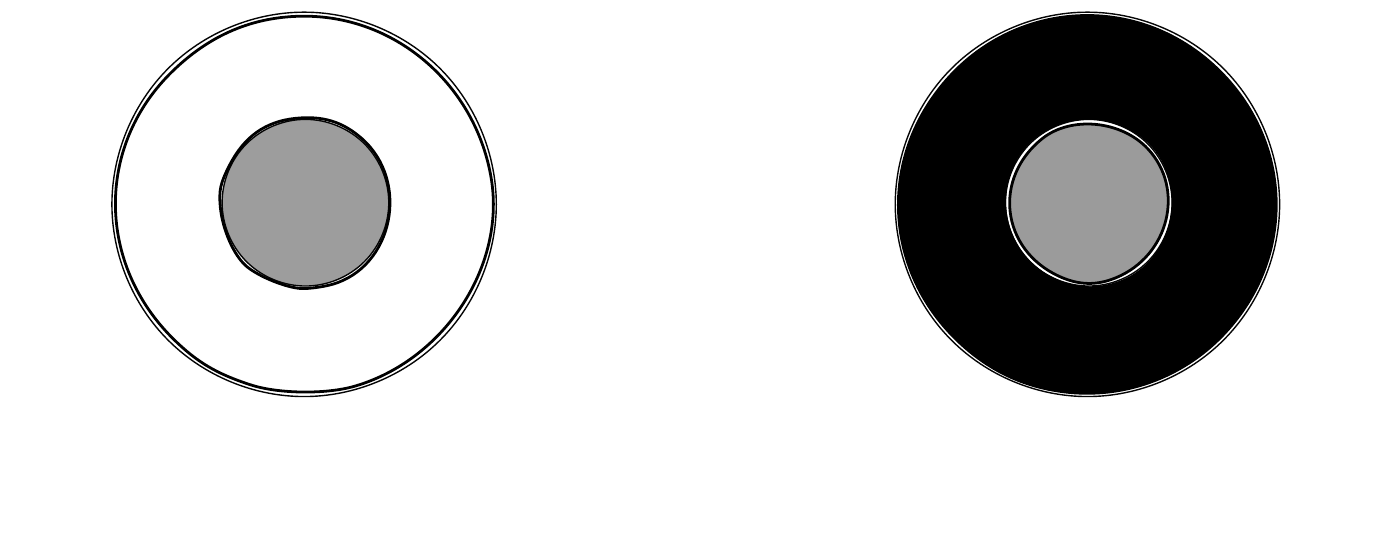}
\vskip -0.2cm
\caption{The phase $f_1$ is black and the phase $f_2$ is white. The mixing of the two phases is represented by the grey color.
The  cartoon on the left represents the unique minimizer in (i). In this case, the two phases mix each other in the inner ball, and the remainig mass of $f_2$ is arranged in an annulus around such ball. In the case (ii), the minimizer has the same form, but replacing $f_2$ (white) with $f_1$ (black).}
\end{figure}

\begin{proof}
We prove only (i) since the proof of (ii) is analogous. Let $(f_1,f_2)$ be a minimizer of $\E^{{c_{11}},{c_{22}}}_{K_{C_N}}$ in $\Am$.
We first notice that, in the case ${c_{22}}<0$, by  i) and ii) of Proposition \ref{armonicofond} we have $f_1 = a_1 \chi_{A}$, $f_2 = a_2 \chi_{A} + \chi_{B}$ for some measurable sets $A, \, B\subset \R^N$. Then, one can argue as in the proof of Corollary \ref{quadr1d} (applying Corollary \ref{triang} instead of Proposition \ref{triangolino1d}). The details are left to the reader. 

It remains to prove the claim for ${c_{22}}=0$. In this case set $\tilde m_1:=\frac{{c_{11}}+2}{2}m_1$ and $\tilde m_2:=-\frac {c_{11}} 2 m_1+m_2$. By assumption $\tilde m_2>\tilde m_1$. Set moreover
\begin{equation}\label{formulag}
g_1:=\frac{{c_{11}}+2}{2}f_1,\qquad g_2:=-\frac{{c_{11}}}{2}f_1+f_2.
\end{equation}
It is easy to see that $g_i\ge 0$,  $\int_{\R^N}g_i(x)\ud x=\tilde m_i$ (for $i=1,2$) and $g_1+g_2=f_1+f_2\le 1$, so that $(g_1,g_2)\in \mathcal A_{\tilde m_1,\tilde m_2}$. Moreover, a straightforward computation yields 
$$
\mathcal E_{K_{C_N}}^{{c_{11}},0}(f_1,f_2)=\frac{2}{{c_{11}}+2}\,\mathcal E_{K_{C_N}}^{0,0}(g_1,g_2),
$$
and hence $(f_1,f_2)$ is a minimizer of $\mathcal E_{K_{C_N}}^{{c_{11}},0}$ in $\Am$ if and only if $(g_1,g_2)$ is a minimizer of $\mathcal E_{K_{C_N}}^{0,0}$ in $ \mathcal A_{\tilde m_1,\tilde m_2}$. By Theorem \ref{00}, the unique (up to a translation) minimizer of $\mathcal E_{K_{C_N}}^{0,0}$ in $ \mathcal A_{\tilde m_1,\tilde m_2}$ is given by $(g_1,g_2)=(\frac 1 2 \chi_{B^{2\tilde m_1}},\frac 1 2 \chi_{B^{2\tilde m_1}}+\chi_{B^{\tilde m_1+\tilde m_2}\setminus B^{\tilde m_1}})$. This, together with \eqref{formulag},  concludes the proof. 
\end{proof}

\section*{Conclusions and perspectives}

We have studied existence and qualitative properties of minimizers of the energy  
$$
\E_{K}^{c_{11},c_{22}}(f_1,f_2)=c_{11}\,J_{K}(f_1,f_1)+c_{22}\,J_{K}(f_2,f_2)-2 J_{K}(f_1,f_2),
$$
in the class of densities $(f_1,f_2)\in L^{1}(\R^N;[0,1])\times  L^{1}(\R^N;[0,1])$ with fixed masses $m_1,m_2$ and satisfying the constraint $f_1+f_2\le 1$.
We have focused on  the attractive case   $c_{11},c_{22}\le 0$  (the checkerboard region in Figure \ref{figconc}), and proved the existence of a minimizer in this case 
for all the values of masses $m_1,m_2$ (see Theorem \ref{esist}). 
Moreover, for  $0<c_{11}=c_{22}\le 1$, $m_1=m_2$ and $K$  positive definite (the dashed segment in Figure \ref{figconc}), we have proved 
that there exists a minimizer (see Lemma \ref{fact} and Remark \ref{nonex}). 
Finally, for $c_{11},c_{22}\ge 1$ with $\max\{c_{11},c_{22}\}>1$ (grey region in the Figure \ref{figconc}),  the energy $\E_{K}^{c_{11},c_{22}}$ does not admit a minimizer for any pair of values $m_1$ and $m_2$ (see Remark \ref{nonex}). 

\begin{figure}[here]
\centering
{\def\svgwidth{150pt}
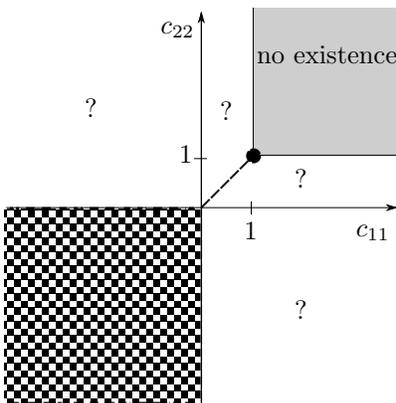}
\caption{Existence/Non existence regions of parameters $c_{11},\, c_{22}$}\label{figconc}
\end{figure}

A natural question arising from these (partial) results is whether existence of minimizers can be proven in the remaining cases. A general existence result, i.e., independent of the masses, seems to be false if at least one of the coefficients is strictly positive. Indeed,   the corresponding phase would loose some of its (if too large)  mass.  In this case, existence results depending on the masses seems to be an interesting issue. 

A relevant aspect of our analysis is that, for the Coulomb interaction kernel, we have found the explicit shape of minimizers for all choices of negative coefficients, expect when they are both strictly less than $-1$ (see Figure \ref{Figcul}). In this case, we can still say that $f_i$ are characteristic functions of two pairwise disjoint sets. But their specific shape is unknown, and could be analyzed  using numerical methods. 
 \begin{figure}[here]
\centering
{\def\svgwidth{300pt}
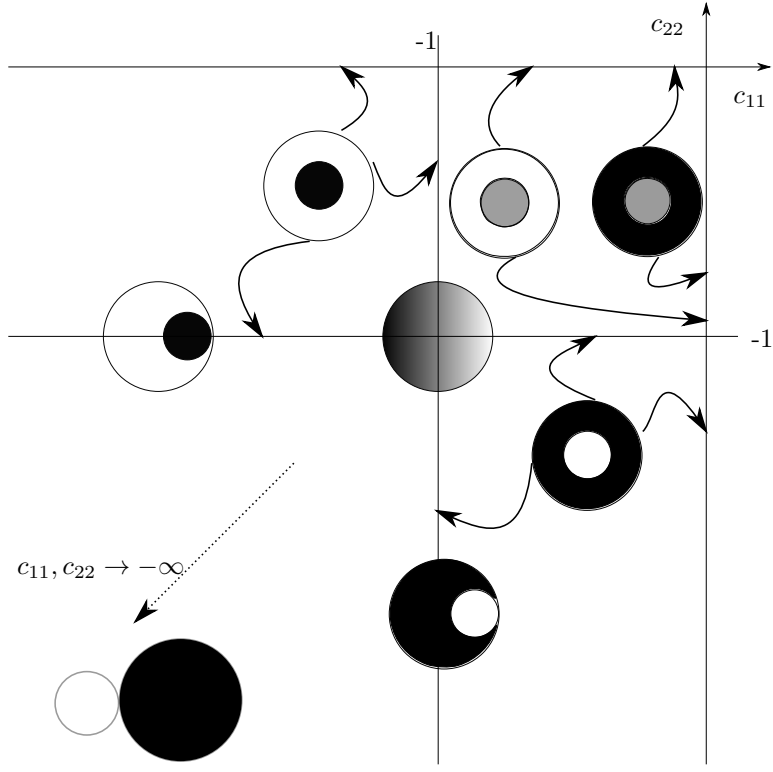}
\caption{Minimizers for Coulomb interactions}\label{Figcul}
\end{figure}

For general kernels  our analysis is far from being complete. 
Nevertheless, there are many possible generalizations we would like to comment on. 
  
First of all, 
one may study the minimum problem above for some specific kernels that are very used in the context of population dynamics (see for instance \cite{CDFLS,DF} and the references therein) such as,  Gaussian, Morse or power law kernels, or suitable combinations of these ones. Moreover, one might remove the assumption that the cross and self interaction kernels $K_{ij}$ are all multiples of a given $K$. Actually, it would be interesting also to understand whether the improved Riesz inequality established in Corollary \ref{improvedRiesz} holds true for more general kernels.  We notice that this Corollary is equivalent to Theorem \ref{00} once one knows that there is not coexistence of two homogeneous phases, i.e., when $f_i$ are  as in Lemma \ref{solounaomogenea}.

Another interesting direction is the extension of the model to the case of $n$ species, i.e., considering minimizers of functionals of the type

$$
\E_K(f_1,\ldots,f_n):=\sum_{i,j=1}^{n}J_{K_{ij}}(f_i,f_j)
$$
under the constraint $\sum_{i=1}^{n}f_i\le 1$ and $\int_{\R^N} f_{i}(x)\ud x=m_{i}$ for $i=1,2,\dots,n$. We believe that some of the techniques developed here could be slightly modified in order to prove existence and some qualitative properties  of the minimizers. 
As already mentioned, the explicit shape of minimizers might require a specific  analysis and could be subject of numerical investigation.

Finally, we point out that our analysis focuses only on the global minimizers of the functional $\Eab$.
Notice that  ground states play a crucial role in the long time asymptotics of nonlinear aggregation-diffusion models. Nevertheless,  the analysis of stationary  states (rather than minimizers) would provide a  better understanding of such problems. In this respect, an interesting analysis would concern   the dynamics of two phases governed  by 
the   energy proposed in this paper. A suitable notion of Wasserstein gradient flow could be considered, in the spirit of \cite{CDFLS,MRSV}.

\end{document}

%% file: tabellagen.pdf_tex
\begingroup%
  \makeatletter%
  \providecommand\color[2][]{%
    \errmessage{(Inkscape) Color is used for the text in Inkscape, but the package 'color.sty' is not loaded}%
    \renewcommand\color[2][]{}%
  }%
  \providecommand\transparent[1]{%
    \errmessage{(Inkscape) Transparency is used (non-zero) for the text in Inkscape, but the package 'transparent.sty' is not loaded}%
    \renewcommand\transparent[1]{}%
  }%
  \providecommand\rotatebox[2]{#2}%
  \ifx\svgwidth\undefined%
    \setlength{\unitlength}{1000bp}%
    \ifx\svgscale\undefined%
      \relax%
    \else%
      \setlength{\unitlength}{\unitlength * \real{\svgscale}}%
    \fi%
  \else%
    \setlength{\unitlength}{\svgwidth}%
  \fi%
  \global\let\svgwidth\undefined%
  \global\let\svgscale\undefined%
  \makeatother%
  \begin{picture}(1,0.584)%
    \put(0,0){\includegraphics[width=\unitlength]{tabellagen.pdf}}%
    \put(0.02999767,0.43926812){\color[rgb]{0,0,0}\makebox(0,0)[lb]{\smash{{\tiny $-1<c_{11}\le 0$}}}}%
    \put(0.04550609,0.26316333){\color[rgb]{0,0,0}\makebox(0,0)[lb]{\smash{{\tiny $c_{11}=-1$}}}}%
    \put(0.04710609,0.08876333){\color[rgb]{0,0,0}\makebox(0,0)[lb]{\smash{{\tiny $c_{11}<-1$}}}}%
    \put(0.24390608,0.55116331){\color[rgb]{0,0,0}\makebox(0,0)[lb]{\smash{{\tiny $-1<c_{22}\le 0$}}}}%
    \put(0.5303061,0.55116331){\color[rgb]{0,0,0}\makebox(0,0)[lb]{\smash{{\tiny $c_{22}=-1$}}}}%
    \put(0.8007061,0.55116331){\color[rgb]{0,0,0}\makebox(0,0)[lb]{\smash{{\tiny $c_{22}<-1$}}}}%
    \put(0.56221807,0.44191792){\color[rgb]{0,0,0}\makebox(0,0)[lb]{\smash{?}}}%
    \put(0.28710608,0.26956333){\color[rgb]{0,0,0}\makebox(0,0)[lb]{\smash{?}}}%
    \put(0.1639061,0.38316333){\color[rgb]{0,0,0}\makebox(0,0)[lb]{\smash{{\tiny if $(c_{11}+1)m_1=(c_{22}+1)m_2$}}}}%
    \put(0.7863061,0.38316333){\color[rgb]{0,0,0}\makebox(0,0)[lb]{\smash{{\tiny if $c_{11}+c_{22}\le-2$}}}}%
    \put(0.24230608,0.02956333){\color[rgb]{0,0,0}\makebox(0,0)[lb]{\smash{{\tiny if $c_{11}+c_{22}\le-2$}}}}%
    \put(0.03850907,0.55146355){\color[rgb]{0,0,0}\makebox(0,0)[lb]{\smash{{\tiny general $K$}}}}%
    \put(0.8087061,0.02636333){\color[rgb]{0,0,0}\makebox(0,0)[lb]{\smash{{\tiny if $N=1$}}}}%
    \put(0.79033545,0.05185522){\color[rgb]{0,0,0}\makebox(0,0)[lb]{\smash{{\tiny $m_1$}}}}%
    \put(0.89019561,0.13082114){\color[rgb]{0,0,0}\makebox(0,0)[lb]{\smash{{\tiny $m_2$}}}}%
    \put(0.1639061,0.36076333){\color[rgb]{0,0,0}\makebox(0,0)[lb]{\smash{{\tiny and $K$ positive definite}}}}%
  \end{picture}%
\endgroup%

%% file: tabellacoul.pdf_tex
\begingroup%
  \makeatletter%
  \providecommand\color[2][]{%
    \errmessage{(Inkscape) Color is used for the text in Inkscape, but the package 'color.sty' is not loaded}%
    \renewcommand\color[2][]{}%
  }%
  \providecommand\transparent[1]{%
    \errmessage{(Inkscape) Transparency is used (non-zero) for the text in Inkscape, but the package 'transparent.sty' is not loaded}%
    \renewcommand\transparent[1]{}%
  }%
  \providecommand\rotatebox[2]{#2}%
  \ifx\svgwidth\undefined%
    \setlength{\unitlength}{1000bp}%
    \ifx\svgscale\undefined%
      \relax%
    \else%
      \setlength{\unitlength}{\unitlength * \real{\svgscale}}%
    \fi%
  \else%
    \setlength{\unitlength}{\svgwidth}%
  \fi%
  \global\let\svgwidth\undefined%
  \global\let\svgscale\undefined%
  \makeatother%
  \begin{picture}(1,0.56)%
    \put(0,0){\includegraphics[width=\unitlength]{tabellacoul.pdf}}%
    \put(0.03201472,0.41446157){\color[rgb]{0,0,0}\makebox(0,0)[lb]{\smash{{\tiny $-1<c_{11}\le 0$}}}}%
    \put(0.04752314,0.25435679){\color[rgb]{0,0,0}\makebox(0,0)[lb]{\smash{{\tiny $c_{11}=-1$}}}}%
    \put(0.04752314,0.09435679){\color[rgb]{0,0,0}\makebox(0,0)[lb]{\smash{{\tiny $c_{11}<-1$}}}}%
    \put(0.24592312,0.51675678){\color[rgb]{0,0,0}\makebox(0,0)[lb]{\smash{{\tiny $-1<c_{22}\le 0$}}}}%
    \put(0.53072314,0.51675678){\color[rgb]{0,0,0}\makebox(0,0)[lb]{\smash{{\tiny $c_{22}=-1$}}}}%
    \put(0.8011231,0.51675678){\color[rgb]{0,0,0}\makebox(0,0)[lb]{\smash{{\tiny $c_{22}<-1$}}}}%
    \put(0.8091231,0.02875679){\color[rgb]{0,0,0}\makebox(0,0)[lb]{\smash{{\tiny if $N=1$}}}}%
    \put(0.79075249,0.05424868){\color[rgb]{0,0,0}\makebox(0,0)[lb]{\smash{{\tiny $m_1$}}}}%
    \put(0.8906127,0.1332146){\color[rgb]{0,0,0}\makebox(0,0)[lb]{\smash{{\tiny $m_2$}}}}%
    \put(0.04212612,0.51545703){\color[rgb]{0,0,0}\makebox(0,0)[lb]{\smash{{\tiny $K$ Coulomb}}}}%
  \end{picture}%
\endgroup%

%% file: proptutto.pdf_tex
\begingroup%
  \makeatletter%
  \providecommand\color[2][]{%
    \errmessage{(Inkscape) Color is used for the text in Inkscape, but the package 'color.sty' is not loaded}%
    \renewcommand\color[2][]{}%
  }%
  \providecommand\transparent[1]{%
    \errmessage{(Inkscape) Transparency is used (non-zero) for the text in Inkscape, but the package 'transparent.sty' is not loaded}%
    \renewcommand\transparent[1]{}%
  }%
  \providecommand\rotatebox[2]{#2}%
  \ifx\svgwidth\undefined%
    \setlength{\unitlength}{560bp}%
    \ifx\svgscale\undefined%
      \relax%
    \else%
      \setlength{\unitlength}{\unitlength * \real{\svgscale}}%
    \fi%
  \else%
    \setlength{\unitlength}{\svgwidth}%
  \fi%
  \global\let\svgwidth\undefined%
  \global\let\svgscale\undefined%
  \makeatother%
  \begin{picture}(1,0.28571429)%
    \put(0,0){\includegraphics[width=\unitlength]{proptutto.pdf}}%
    \put(0.10916082,0.01159702){\color[rgb]{0,0,0}\makebox(0,0)[lb]{\smash{{\tiny (i)}}}}%
    \put(0.48080989,0.01007254){\color[rgb]{0,0,0}\makebox(0,0)[lb]{\smash{{\tiny (ii)}}}}%
    \put(0.86934928,0.00829555){\color[rgb]{0,0,0}\makebox(0,0)[lb]{\smash{{\tiny (iii)}}}}%
  \end{picture}%
\endgroup%

%% file: propmass.pdf_tex
\begingroup%
  \makeatletter%
  \providecommand\color[2][]{%
    \errmessage{(Inkscape) Color is used for the text in Inkscape, but the package 'color.sty' is not loaded}%
    \renewcommand\color[2][]{}%
  }%
  \providecommand\transparent[1]{%
    \errmessage{(Inkscape) Transparency is used (non-zero) for the text in Inkscape, but the package 'transparent.sty' is not loaded}%
    \renewcommand\transparent[1]{}%
  }%
  \providecommand\rotatebox[2]{#2}%
  \ifx\svgwidth\undefined%
    \setlength{\unitlength}{132bp}%
    \ifx\svgscale\undefined%
      \relax%
    \else%
      \setlength{\unitlength}{\unitlength * \real{\svgscale}}%
    \fi%
  \else%
    \setlength{\unitlength}{\svgwidth}%
  \fi%
  \global\let\svgwidth\undefined%
  \global\let\svgscale\undefined%
  \makeatother%
  \begin{picture}(1,1)%
    \put(0,0){\includegraphics[width=\unitlength]{propmass.pdf}}%
  \end{picture}%
\endgroup%

%% file: triang.pdf_tex
\begingroup%
  \makeatletter%
  \providecommand\color[2][]{%
    \errmessage{(Inkscape) Color is used for the text in Inkscape, but the package 'color.sty' is not loaded}%
    \renewcommand\color[2][]{}%
  }%
  \providecommand\transparent[1]{%
    \errmessage{(Inkscape) Transparency is used (non-zero) for the text in Inkscape, but the package 'transparent.sty' is not loaded}%
    \renewcommand\transparent[1]{}%
  }%
  \providecommand\rotatebox[2]{#2}%
  \ifx\svgwidth\undefined%
    \setlength{\unitlength}{400bp}%
    \ifx\svgscale\undefined%
      \relax%
    \else%
      \setlength{\unitlength}{\unitlength * \real{\svgscale}}%
    \fi%
  \else%
    \setlength{\unitlength}{\svgwidth}%
  \fi%
  \global\let\svgwidth\undefined%
  \global\let\svgscale\undefined%
  \makeatother%
  \begin{picture}(1,0.4)%
    \put(0,0){\includegraphics[width=\unitlength]{triang.pdf}}%
  \end{picture}%
\endgroup%

%% file: minimoinquadrato.pdf_tex
\begingroup%
  \makeatletter%
  \providecommand\color[2][]{%
    \errmessage{(Inkscape) Color is used for the text in Inkscape, but the package 'color.sty' is not loaded}%
    \renewcommand\color[2][]{}%
  }%
  \providecommand\transparent[1]{%
    \errmessage{(Inkscape) Transparency is used (non-zero) for the text in Inkscape, but the package 'transparent.sty' is not loaded}%
    \renewcommand\transparent[1]{}%
  }%
  \providecommand\rotatebox[2]{#2}%
  \ifx\svgwidth\undefined%
    \setlength{\unitlength}{400bp}%
    \ifx\svgscale\undefined%
      \relax%
    \else%
      \setlength{\unitlength}{\unitlength * \real{\svgscale}}%
    \fi%
  \else%
    \setlength{\unitlength}{\svgwidth}%
  \fi%
  \global\let\svgwidth\undefined%
  \global\let\svgscale\undefined%
  \makeatother%
  \begin{picture}(1,0.4)%
    \put(0,0){\includegraphics[width=\unitlength]{minimoinquadrato.pdf}}%
    \put(0.20482513,0.03223583){\color[rgb]{0,0,0}\makebox(0,0)[lb]{\smash{{\tiny (i)}}}}%
    \put(0.77282513,0.03223583){\color[rgb]{0,0,0}\makebox(0,0)[lb]{\smash{{\tiny (ii)}}}}%
  \end{picture}%
\endgroup%

%% file: existence.pdf_tex
\begingroup%
  \makeatletter%
  \providecommand\color[2][]{%
    \errmessage{(Inkscape) Color is used for the text in Inkscape, but the package 'color.sty' is not loaded}%
    \renewcommand\color[2][]{}%
  }%
  \providecommand\transparent[1]{%
    \errmessage{(Inkscape) Transparency is used (non-zero) for the text in Inkscape, but the package 'transparent.sty' is not loaded}%
    \renewcommand\transparent[1]{}%
  }%
  \providecommand\rotatebox[2]{#2}%
  \ifx\svgwidth\undefined%
    \setlength{\unitlength}{488bp}%
    \ifx\svgscale\undefined%
      \relax%
    \else%
      \setlength{\unitlength}{\unitlength * \real{\svgscale}}%
    \fi%
  \else%
    \setlength{\unitlength}{\svgwidth}%
  \fi%
  \global\let\svgwidth\undefined%
  \global\let\svgscale\undefined%
  \makeatother%
  \begin{picture}(1,1)%
    \put(0,0){\includegraphics[width=\unitlength]{existence.pdf}}%
    \put(0.20174285,0.72532627){\color[rgb]{0,0,0}\makebox(0,0)[lb]{\smash{?}}}%
    \put(0.73114754,0.21639342){\color[rgb]{0,0,0}\makebox(0,0)[lb]{\smash{?}}}%
    \put(0.63606557,0.86033223){\color[rgb]{0,0,0}\makebox(0,0)[lb]{\smash{{no existence}}}}%
    \put(0.54272646,0.71876889){\color[rgb]{0,0,0}\makebox(0,0)[lb]{\smash{?}}}%
    \put(0.60413173,0.41815924){\color[rgb]{0,0,0}\makebox(0,0)[lb]{\smash{1}}}%
    \put(0.44025063,0.60938609){\color[rgb]{0,0,0}\makebox(0,0)[lb]{\smash{1}}}%
    \put(0.73114754,0.54754096){\color[rgb]{0,0,0}\makebox(0,0)[lb]{\smash{?}}}%
    \put(0.88575199,0.42403019){\color[rgb]{0,0,0}\makebox(0,0)[lb]{\smash{$c_{11}$}}}%
    \put(0.39340583,0.93522221){\color[rgb]{0,0,0}\makebox(0,0)[lb]{\smash{$c_{22}$}}}%
  \end{picture}%
\endgroup%

%% file: coulombcart.pdf_tex
\begingroup%
  \makeatletter%
  \providecommand\color[2][]{%
    \errmessage{(Inkscape) Color is used for the text in Inkscape, but the package 'color.sty' is not loaded}%
    \renewcommand\color[2][]{}%
  }%
  \providecommand\transparent[1]{%
    \errmessage{(Inkscape) Transparency is used (non-zero) for the text in Inkscape, but the package 'transparent.sty' is not loaded}%
    \renewcommand\transparent[1]{}%
  }%
  \providecommand\rotatebox[2]{#2}%
  \ifx\svgwidth\undefined%
    \setlength{\unitlength}{800bp}%
    \ifx\svgscale\undefined%
      \relax%
    \else%
      \setlength{\unitlength}{\unitlength * \real{\svgscale}}%
    \fi%
  \else%
    \setlength{\unitlength}{\svgwidth}%
  \fi%
  \global\let\svgwidth\undefined%
  \global\let\svgscale\undefined%
  \makeatother%
  \begin{picture}(1,1)%
    \put(0,0){\includegraphics[width=\unitlength]{coulombcart.pdf}}%
    \put(0.9345603,0.85480573){\color[rgb]{0,0,0}\makebox(0,0)[lb]{\smash{$c_{11}$}}}%
    \put(0.83026587,0.94478527){\color[rgb]{0,0,0}\makebox(0,0)[lb]{\smash{$c_{22}$}}}%
    \put(0.53374231,0.92229039){\color[rgb]{0,0,0}\makebox(0,0)[lb]{\smash{-1}}}%
    \put(0.95705524,0.54601227){\color[rgb]{0,0,0}\makebox(0,0)[lb]{\smash{-1}}}%
    \put(0.03097342,0.2596687){\color[rgb]{0,0,0}\makebox(0,0)[lb]{\smash{$c_{11},c_{22}\to -\infty$}}}%
  \end{picture}%
\endgroup%